\documentclass[12pt]{article}
\usepackage{amssymb,amsthm,amsmath,amsfonts, fullpage, ytableau, shuffle,enumitem,multicol, tikz,xcolor}
\usepackage{mycommands}
\usepackage{hyperref}
\usepackage{authblk}

\usetikzlibrary{patterns,arrows,decorations.pathreplacing}

\theoremstyle{plain}
\newtheorem{thm}{Theorem}[section]

\newtheorem{lemma}[thm]{Lemma}
\newtheorem{fact}[thm]{Fact}

\theoremstyle{definition}
\newtheorem{remark}[thm]{Remark}
\newtheorem{definition}[thm]{Definition}
\newtheorem{notation}[thm]{Notation}
\newtheorem{definitions}[thm]{Definitions}

\DeclareMathOperator{\lis}{lis}
\DeclareMathOperator{\lds}{lds}
\DeclareMathOperator{\iRSK}{iRSK}
\DeclareMathOperator{\fus}{fus}
\DeclareMathOperator{\fjdt}{Kjdt}
\DeclareMathOperator{\rjdt}{Krjdt}

\renewcommand{\Vert}{\textrm{Vert}}
\newcommand{\Horiz}{\textrm{Horiz}}
\newcommand{\Rect}{\textrm{Rect}}
\newcommand{\Inc}{\textrm{Inc}}
\newcommand{\W}[2]{{\cW}_{#1}^{({#2})}}
\newcommand{\kW}[3]{{\cW_{{#1}}^{({#2},{#3})}}}
\newcommand{\Wcoor}[2]{\langle #1,#2\rangle}

\newcommand{\mystar}{{\scalebox{1.1}{\ensuremath \star}}}

\makeatletter
\newcommand*\bigcdot{\mathpalette\bigcdot@{1}}
\newcommand*\bigcdot@[2]{\mathbin{\vcenter{\hbox{\scalebox{#2}{$\m@th#1\bullet$}}}}}
\makeatother

\begin{document}
\pagestyle{plain}

\title{Strictly increasing and decreasing sequences in subintervals of words and a conjecture of Guo and Poznanovi\'c}
\date{}
\author[1]{Jonathan S. Bloom}
\author[2]{Dan Saracino}
\affil[1]{Department of Mathematics, Lafayette College, Easton PA, USA}
\affil[2]{Department of Mathematics, Colgate University, Hamilton NY, USA}
\maketitle

\begin{abstract}
\sloppy
We prove a conjecture of Guo and Poznanovi\'{c} concerning chains in certain 01-fillings of moon polyominoes. A key ingredient of our proof is a correspondence between words $w$ and pairs $(\cW(w), \cM(w))$ of increasing tableaux such that $\cM(w)$ determines the lengths of the longest strictly increasing and strictly decreasing sequences in every subinterval of $w$. We define this correspondence by using Thomas and Yong's K-infusion operator and then use it to obtain the bijections that prove the conjecture of Guo and Poznanovi\'{c}.  In constructing our bijections we introduce new variants of the RSK correspondence and Knuth equivalence. 
\end{abstract}

\section{Introduction}\label{sec:intro}
The motivation for the work presented in this paper was to prove a conjecture of Guo and Poznanovi\'{c} (\cite{GuoPoznanovik2020}, Conjecture 15) about chains in certain 01-fillings of moon polyominoes. We state the conjecture after recalling some required definitions, and we prove the conjecture in Section 5.

Let $S$ be the set of square boxes in the $xy$-plane determined by the points of $\bbZ\times\bbZ$.  A \emph{polyomino} is a finite subset of $S$.  A polyomino $M$ is a \emph{moon polyomino} if it contains every box situated between any two of its boxes that are in the same row or column, and for any two of its rows (or columns), one is a subset of the other, up to translation.  Adopting the terminology of Guo and Poznanovi\'{c}, we call a moon polyomino a \emph{stack polyomino} if its rows are left justified.   (We note that some authors use the phrase \emph{stack polyomino} to refer to a moon polyomino whose columns are bottom justified.) A \emph{Ferrers board} is a stack polyomino in which the lengths of the rows are weakly decreasing from top to bottom.

\begin{figure}
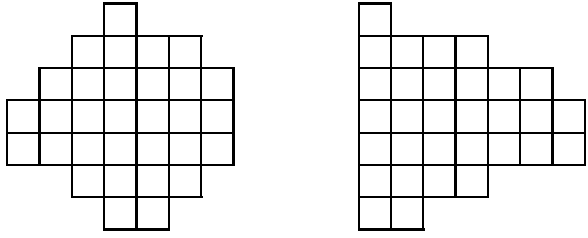

\ytableausetup{ boxsize=1em}
    \centering
        \ydiagram{3+1,2+4,1+6,7,7,2+4,3+2}\qquad\qquad \ydiagram{1,4,6,7,7,4,2}
    \caption{An example of a moon polyomino (left) and a stack polyomino (right).}
    \label{fig:moon and stack}
\end{figure}

By a \emph{01-filling} of a moon polyomino we mean an assignment of either a 0 or a 1 to each box of the polyomino.  We sometimes refer to the boxes containing 0's as \emph{empty}. A sequence of 1's in a filling of $M$ is a \emph{chain} in $M$ if the smallest rectangle containing all the 1's in the sequence is completely contained in $M$. A chain is called a \emph{ne chain} if each 1 in the chain is strictly above and strictly to the right of all the preceding 1's in the chain. In an analogous way, we define \emph{se chains}. The \emph{length} of a chain is the number of $1$'s in the chain. 

 To state our result precisely, we adopt the following notation of Guo and Poznanovi\'{c}. For any moon polyomino $M$ and positive integers $n,u,v$,  we let $N(M;n;ne=u,se=v)$ denote the number of 01-fillings of $M$ with exactly $n$ 1's such that each column contains at most one 1, and the lengths of the longest ne and se chains in $M$ are $u$ and $v$, respectively. Guo and Poznanovi\'{c} proved in \cite{GuoPoznanovik2020}
 that for every stack polyomino $M$ and all positive integers $n,u,v$, the quantity
$N(M;n; ne=u,se=v)$ depends only on the multiset of lengths of the rows of $M$.  They conjectured that their result is in fact true for all moon polyominoes, and we prove their conjecture as

\begin{thm}\label{thm:moon}
For every moon polyomino $M$ and all positive integers $n,u,v$, the quantity $N(M;n; ne=u,se=v)$ depends only on the multiset of lengths of the rows of $M$.
\end{thm}

Note that two moon polyominoes have the same multiset of lengths of rows if and only if they have the same multiset of heights of columns.

Before we set our result in historical context, we want to try to provide a rough idea of what the novelty of our approach to the problem was, and what the obstacles were that we had to overcome in order to implement our approach.  The basic idea of the proof of the result of Guo and Poznanovi\'{c} was to start with a stack polyomino $M$ and transform it to a stack polyomino $M'$ by removing the bottom row $r$ of $M$ and inserting a new row $r'$  (with the same length as $r$) directly above the top row of the largest rectangle in $M$ that contained $r$. (This idea had been used earlier by Rubey in \cite{Rubey2011}.)  Guo and Poznanovi\'{c} showed that this transformation preserves the quantity $N(M;n; ne=u,se=v)$ by constructing a bijection between fillings of $M$ and fillings of $M'$ such that, for every maximal rectangle $X$ in $M$, the greatest length of a ne (respectively, se) chain in $X$ is the same as the greatest length of a ne (respectively se) chain in the maximal rectangle $X'$ in $M'$ corresponding to $X$. (By ``corresponding to" we mean ``having the same dimensions as.")    By iterating this idea they transformed the stack polyomino into a Ferrers board (in French notation, with the longest rows on the bottom) while preserving $N(M;n; ne=u,se=v)$.  Since a Ferrers board in French notation is completely determined by the multiset of its row lengths, this proved their theorem.

A natural way to try to extend their result to moon polyominoes was to transform a moon polyomino $M$ to a stack polyomino $M'$ by successively removing the leftmost column $c$ and adding a new column $c'$ directly to the right of the largest rectangle 
$R$ in $M$ containing $c$. The challenge was to find a bijection between fillings of $M$ and fillings of $M'$ that preserves the maximum length of ne and se chains in corresponding maximal rectangles $X$ and $X'$ in $M$ and $M'$. We hoped, as in \cite{Rubey2011} and \cite{GuoPoznanovik2020}, to do so by changing a given filling of $M$ only by using the filling of $R$ to create a new filling of the corresponding rectangle $R'$ in $M$.

The nonempty rows and columns of the filling of $R$ give us the matrix representation of a word $w$, and we wanted to find a word $w'$ with the same length $n$ and the same set of values as $w$ that would similarly provide us the nonempty rows and columns of the new filling of $R'$. Consider, for example, what is required of $w'$ in the case of a maximal rectangle $X$ in $M$ that is narrower than $R$, and a filling of $M$ such that column $c$ is not empty. We want to place the values of $w'$ in the columns of $R'$ that were not empty in $R$, and $c'$. For an illustration of such an $R, X$ and $c$, see Figure 2, where, in $M$, $R$ consists of the boxes colored dark blue, light blue, or violet, $X$ consists of the boxes colored red or violet, and $c$ consists of the boxes colored dark blue.

		
		

\begin{figure}[h!]
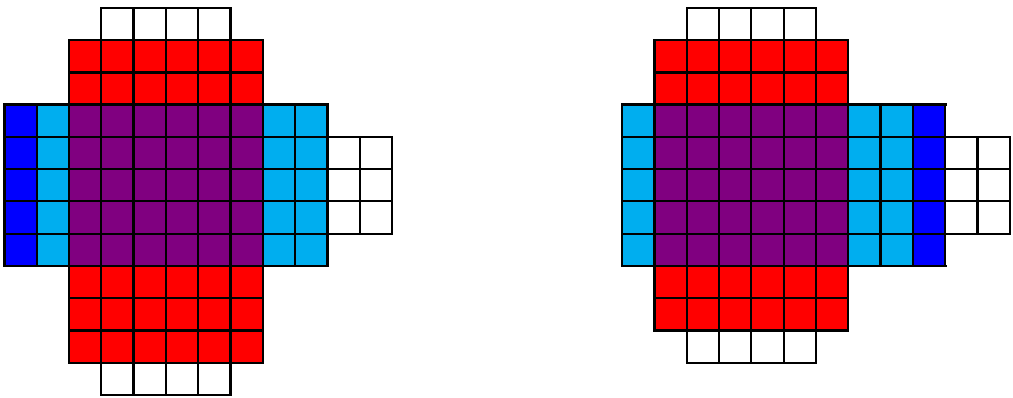

\ytableausetup{ boxsize=1em}
    \centering
    \begin{ytableau}
         \none & \none &\none&&&&\\
         \none &\none &*(red) &*(red) &*(red)&*(red)&*(red)&*(red)\\
         \none &\none &*(red) &*(red) &*(red)&*(red)&*(red)&*(red)\\
        *(blue)&*(cyan) &*(violet) &*(violet) &*(violet) &*(violet)&*(violet)&*(violet)&*(cyan)&*(cyan)\\
        *(blue)& *(cyan)&*(violet) &*(violet) &*(violet) &*(violet)&*(violet)&*(violet)&*(cyan)&*(cyan)&& \\
        *(blue)& *(cyan)&*(violet) &*(violet) &*(violet) &*(violet)&*(violet)&*(violet)&*(cyan)&*(cyan)&& \\
        *(blue)& *(cyan)&*(violet) &*(violet) &*(violet) &*(violet)&*(violet)&*(violet)&*(cyan)&*(cyan)&& \\
        *(blue)&*(cyan) &*(violet) &*(violet) &*(violet) &*(violet)&*(violet)&*(violet)&*(cyan)&*(cyan) \\
        \none&\none &*(red)&*(red)&*(red)&*(red)&*(red)&*(red)\\
        \none&\none&*(red)&*(red)&*(red)&*(red)&*(red)&*(red)\\
        \none&\none&*(red)&*(red)&*(red)&*(red)&*(red)&*(red)\\
        \none&\none&\none&&&&
    \end{ytableau}
    \qquad\qquad\qquad
\begin{ytableau}
         \none & \none &\none&&&&\\
         \none &\none &*(red) &*(red) &*(red)&*(red)&*(red)&*(red)\\
         \none &\none & *(red)& *(red)&*(red)&*(red)&*(red)&*(red)\\
        \none&*(cyan) &*(violet) &*(violet) &*(violet) &*(violet)&*(violet)&*(violet)&*(cyan)&*(cyan)&*(blue)\\
        \none& *(cyan)&*(violet) &*(violet) &*(violet) &*(violet)&*(violet)&*(violet)&*(cyan)&*(cyan)&*(blue)&& \\
        \none& *(cyan)&*(violet) &*(violet) &*(violet) &*(violet)&*(violet)&*(violet)&*(cyan)&*(cyan)&*(blue)&& \\
        \none& *(cyan)&*(violet) &*(violet) &*(violet) &*(violet)&*(violet)&*(violet)&*(cyan)&*(cyan)&*(blue)&& \\
        \none&*(cyan) &*(violet) &*(violet) &*(violet) &*(violet)&*(violet)&*(violet)&*(cyan)&*(cyan)&*(blue) \\
        \none&\none &*(red)&*(red)&*(red)&*(red)&*(red)&*(red)\\
        \none&\none&*(red)&*(red)&*(red)&*(red)&*(red)&*(red)\\
        \none&\none&\none&&&&
    \end{ytableau}
    \caption{Moon polyominoes $M$ (left) and $M'$ (right)}
    \label{fig:M and M'}
\end{figure}
Any ne chain of maximal length in X will consist of initial and terminal segments outside $R$ and a middle segment in some interval $J$ of the columns of $R$, and the middle segment will have maximal length among the ne chains in the interval $J$ of the columns of $R$. Our motivating observation was that what this requires of $w'$ is that, for any subinterval $[i,j]$ of $[2,n]$, the maximum length of a ne chain should be the same in $w'_{i-1}\cdots w'_{j-1}$ as it is in $w_i\cdots w_j$.  Of course, we require the same for se chains.  And when we try to deal with maximal rectangles that are wider than $R$, it turns out that, in terms of the matrix representations of $w$ and $w'$, we require that the maximum length of a ne (or se) chain in any subinterval of the rows should be the same for $w'$ as it is for $w$.

We need some terminology to deal with these ideas efficiently.  Let $\fW_n$ be the set of all words of length $n$ whose letters are positive integers.    For any word $u$ define $\lis(u)$ (respectively, $\lds(u)$) to be the length of a longest strictly increasing (respectively, decreasing) subsequence in $u$.  We say two words $u,w\in \fW_n$ have the same \emph{CID data} and write $u\sim_{cid} w$ provided that for any choice of $1\leq i\leq j\leq n$ we have
$$\lis(u_i\cdots u_j) = \lis(w_i\cdots w_j)\myand \lds(u_i\cdots u_j) = \lds(w_i\cdots w_j).$$
(We note that ``CID" stands for ``column increasing decreasing", and ``column" refers to the columns in the matrix representation of a word.) If we only know the first (respectively, second) equality above holds we write $u\sim_{ci} w$ (respectively, $u\sim_{cd} w$). We say $u$ and $w$ have the same \emph{RID data} and write $u\sim_{rid} w$ if, for every integer interval $[a,b]$, we have $\lis(u')=\lis(w')$ and $\lds(u')=\lds(w')$, where $u'$ and $w'$ are the words obtained from $u$ and $w$ by deleting all integers not in $[a,b]$. 

Our proof of Theorem~\ref{thm:moon} rests on developing a correspondence between words and pairs of certain increasing tableaux, such that if $w\mapsto (\cW(w),\cM(w))$ then $\cW(u) = \cW(w)$ implies $u\sim_{rid} w$ and $\cM(u) = \cM(w)$ implies $u\sim_{cid} w$.  We note that the known ways of associating pairs of tableaux to words do not achieve what $\cM(w)$ achieves. For example, under the RSK correspondence the words $u=32311$ and $w=32211$ have the same recording tableau \ytableausetup{boxsize=1.2em}
$$\ytableaushort{13,25,4}$$
but we see that $u\not\sim_{cid} w$ since the 2nd and 3rd columns of $u$ have longest increasing sequence length 2, while in $w$ the same columns have longest increasing sequence length 1. Under the correspondence provided by Hecke insertion (which is what Guo and Poznanovi\'{c} used to define their bijection) the words $u = 13221, w = 23231$  both have the Hecke recording tableau
\ytableausetup{boxsize=1.7em}
$$\begin{ytableau}
1 & 2,4\\
3 & 5
\end{ytableau}$$
but $u\not\sim_{cid} w$ since the 3rd and 4th columns of $u$ have longest increasing sequence length 1, while in $w$ the same columns have longest increasing sequence length 2.

We will obtain our new correspondence $w\mapsto (\cW(w),\cM(w))$ that captures both RID and CID data by using a new technique: the K-infusion operator introduced by Thomas and Yong in \cite{Thomas2007}. We establish our correspondence in Section 2, where we also prove (Theorem~\ref{thm:W implies rid}) that $\cW(u) = \cW(w)$ implies $u\sim_{rid} w$. The proof that
$\cM(u) = \cM(w)$ implies $u\sim_{cid} w$ requires new ideas and considerably more effort, and we defer it to Section 4.  In particular, we need to introduce and work with a new variant of RSK, which we call iRSK, and a new variant of Knuth equivalence, which we call i-Knuth equivalence, to prove that $\cM(w)$ has the desired property.

In Section 3 we use the correspondence $w\mapsto (\cW(w),\cM(w))$ to construct two bijections. One of these 
(in Theorem~\ref{thm:col shift}) enables us to prove Theorem~\ref{thm:moon}, by providing a bijection from $\fW_n$ to $\fW_n$ that preserves RID data but ``shifts CID data leftward," as was required in the above illustration.
The other bijection (in Theorem~\ref{thm:row shift}) enables us to prove the result of Guo and Poznanovi\'{c} for stack polyominoes in a new way, by providing a bijection from $\fW_n$ to $\fW_n$ that preserves CID data but ``shifts RID data downward." Theorems~\ref{thm:row shift} and \ref{thm:col shift}, taken together, provide a unified proof of Theorem~\ref{thm:moon} independent of the result of Guo and Poznanovi\'{c}.

Finally, in Section 5 we put everything together to prove the conjecture of Guo and Poznanovi\'{c}.

Before we turn to the construction of $\cW(w)$ and $\cM(w)$, we provide the following (incomplete) account of previous results, to set our results in context.

A number of authors have contributed results about the enumeration of fillings of polyominoes with restrictions on the maximum lengths of chains of various types.  Specializing to 01-fillings with at most one 1 in each row and at most one 1 in each column (which can be dealt with by considering permutations), Krattenthaler~\cite{Krattenthaler2006} proved the following: For any Ferrers board $F$ and any positive integers $n,s,t$, the number of such fillings of $F$ with exactly $n$ 1's and longest ne and se chains of lengths $s$ and $t$, respectively, equals the number of such fillings of $F$ with exactly $n$ 1's and longest ne and se chains of lengths $t$ and $s$, respectively. He then showed that this result implies the main results of \cite{Chen2005} about crossings and nestings in matchings and set partitions, and went on to obtain extensions of his result for general 01-fillings and fillings with nonnegative integer entries.  His motivation for proving these extensions was the result of Jonsson~\cite{JONSSON2005} and Jonsson and Welker~\cite{JonssonWelker2006}, which states that the total number of 01-fillings of a stack polyomino $M$ with exactly $n$ 1's and longest ne chains of length $s$ depends only on the multiset of heights of the columns of $M$.  While Jonsson used a complicated induction argument, and Jonsson and Welker used commutative algebra, Krattenthaler sought to use the growth diagrams of Fomin and variants of the RSK correspondence to obtain their result.  Although he did not completely succeed in doing so, he paved the way for the work of Rubey~\cite{Rubey2011}. 

By supplementing Krattenthaler's approach with the new idea of promotion of fillings of rectangular shapes, Rubey was able to reprove the result of Jonsson and Welker, and indeed to prove its generalization to moon polyominoes, which had been conjectured by Jonsson. Rubey was also able to prove many additional results involving weakly ne and se chains.

In~\cite{Krattenthaler2006}, Problem 2, Krattenthaler asked whether his result about interchanging the lengths of the longest ne and se chains would hold in the context of all 01-fillings, instead of just those fillings having at most one 1 in each row and column.  A counterexample was provided by de Mier in \cite{Mier2006}, and it follows that, in the context of all 01-fillings, the result of Jonsson and Welker for stack polyominoes cannot be strengthened so as to fix the lengths of the longest ne chains \emph{and} the lengths of the longest se chains.

The result of Guo and Poznanovi\'{c} in \cite{GuoPoznanovik2020} showed, however, that the strengthened version of Jonsson and Welker's result for stack polyominoes does hold in the context of 01-fillings with at most one 1 in each column. And Conjecture 15 in  \cite{GuoPoznanovik2020} led to the writing of our paper.

\section{K-theoretic tools}\label{sec:K-tools}
\subsection{K-Knuth equivalence}	
\begin{definition}
	We say two words $u$ and $w$ are \emph{K-Knuth equivalent} and write $u\equiv_K w$ provided that one can be obtained from the other by a finite number of the following moves:
	\begin{align*}
	\cdots bca\cdots \qquad &\leftrightarrow \quad \cdots bac\cdots \qquad (a< b < c)\\
	\cdots acb\cdots \qquad &\leftrightarrow \quad \cdots cab\cdots\qquad (a < b < c)\\
	\cdots a\cdots \qquad &\leftrightarrow \quad \cdots aa\cdots\\
	\cdots aba\cdots \qquad &\leftrightarrow \quad \cdots bab\cdots.\\
	\end{align*}
\end{definition}

 We will need the following two results about K-Knuth equivalence.

 \begin{fact}\label{fact:Knuth on subintervals of rows} (\cite{Buch2013}, Lemma 5.5)
 If $u$ and $w$ are K-Knuth equivalent words and $[a,b]$ is an integer interval, then the words obtained from $u$ and $v$ by deleting all integers not in $[a,b]$ are K-Knuth equivalent.
 \end{fact}
 
 \begin{fact}\label{fact:longest inc and dec} (\cite{Gaetz2016}, Proposition 37)
 If two words are K-Knuth equivalent, then
 the lengths of the longest strictly decreasing subsequences in these words are the same, and the lengths of the longest strictly increasing subsequences are the same.
\end{fact}

\subsection{K-Infusion}\label{subsec:k-infusion}
\ytableausetup{boxsize=1.2em}

Throughout this subsection let  $\lambda\subseteq \mu$ be partitions, which we view as Ferrers boards. We define an \emph{inner corner} of the skew shape $\mu/\lambda$ to be a maximally southeast box in $\lambda$.  We shall also need the notion of an outer corner.  To define this, let $\nu$ be any rectangular partition such that $\nu_1>\mu_1$ and the length of $\nu$ is greater than the length of $\mu$.  Then an \emph{outer corner} of $\mu/ \lambda$ is a maximally northwest box in $\nu/ \mu$.

A filling $T$ of the skew shape $\mu/ \lambda$ with positive integers is called a \emph{semistandard Young tableau} if the entries are weakly increasing along rows and strictly increasing along columns, from top to bottom.  A filling $T$ is called an \emph{increasing tableau} if the entries are strictly increasing along both rows and columns.  We write  $\sh(T) = \mu/\lambda$ and refer to this as the \emph{shape} of $T$.  In the case $\lambda = \emptyset$ we say $T$ has \emph{straight} shape.  We denote by $\Inc(\mu/\lambda)$ the set of all increasing tableaux of shape $\mu/\lambda$.  We denote by $Inc^{st}(\mu/\lambda)$ the set of elements of $Inc(\mu/ \lambda)$ whose set of entries constitute an integer interval $[1,n]$. 
An element of $Inc^{st}(\mu/\lambda)$ is called a \emph{standard Young tableau} if $\lambda = \emptyset$ and each entry appears exactly once.

A specific increasing tableau of interest is the \emph{minimal tableau of shape $\lambda$}, which is the increasing tableau of shape $\lambda$ where each box is filled with the smallest positive integer possible.   For each $n>0$ we define the \emph{staircase shape} to be the partition $\sigma_n = (n,n-1,\ldots,1)$ and define $M_n$ to be the minimal tableau of shape $\sigma_n$.  For example, $M_4$ is the increasing tableau
$$\begin{ytableau}
1 & 2& 3&4\\
2& 3&4\\
3&4\\	
4
\end{ytableau}\ .$$

Next we recall the idea of $K$-theoretic jeu de taquin, which was first introduced and studied in \cite{Thomas2007}.  For any increasing tableau $T$ of shape $\mu/\lambda$ let $B$ be a set of inner corners.  We then define the increasing tableau $\fjdt_B(T)$ to be the tableau given by the following inductive procedure.  First place a $\bullet$ in each of the inner corners in $B$.  Now, identify all the $\bullet$'s that are directly above or to the left of a box containing a 1.  Replace these dots with 1's and replace those 1's, that were below or to the right of a dot, with a dot.  Next, repeat this process for all dots directly above or to the left of a 2, then  for all dots directly above or to the left of a 3, and so on, until no dot has an entry below or to its right. Then delete the dots. For example, consider the following increasing tableau 
$$\begin{ytableau}
\none &\none &\none &\none[\bullet] & 2 & 3\\
\none &\none &1&2& 3 & 6\\
\none &\none & 2 & 4& 5\\
	\none[\bullet] &1 & 3& 6\\
	1&3&5\\
\end{ytableau}$$
where the selected inner corners are denoted with $\bullet$'s.
Then the steps of the $K$-theoretic jeu de taquin procedure are
$$\begin{ytableau}
\none &\none &\none &\bullet & 2 & 3\\
\none &\none &1&2& 3 & 6\\
\none &\none & 2 & 4& 5\\
1 &\bullet & 3& 6\\
\bullet & 3 &5\\
\end{ytableau}
\begin{ytableau}
\none &\none &\none &2 & \bullet & 3\\
\none &\none &1& \bullet & 3 & 6\\
\none &\none & 2 & 4& 5\\
1 &\bullet & 3& 6\\
\bullet & 3 &5\\
\end{ytableau}
\begin{ytableau}
\none &\none &\none &2 & 3 & \bullet\\
\none &\none &1& 3 & \bullet & 6\\
\none &\none & 2 & 4& 5\\
1 &3 & \bullet& 6\\
3 & \bullet &5\\
\end{ytableau}
\begin{ytableau}
\none &\none &\none &2 & 3 & \bullet\\
\none &\none &1& 3 & 5 & 6\\
\none &\none & 2 & 4& \bullet\\
1 &3 & 5& 6\\
3 & 5 &\bullet\\
\end{ytableau}
\begin{ytableau}
\none &\none &\none &2 & 3 & 6\\
\none &\none &1& 3 & 5 & \none[\bullet]\\
\none &\none & 2 & 4& \none[\bullet]\\
1 &3 & 5& 6\\
3 & 5 &\none[\bullet]\\
\end{ytableau}
$$
and $\fjdt_B(T)$ is the increasing tableau on the right, minus the dots.

Naturally, we also have the notion of reverse $K$-theoretic jeu de taquin. Again we start with an increasing tableau $S$ of shape $\mu/\lambda$ with maximal entry $k$.  This time we take $C$ to be a set of outer corners for $S$ and define  $\rjdt_C(S)$ as follows.  First, place a dot $\bullet$ in each outer corner from $C$ and find all the dots that are directly below or to the right of a box containing a $k$.  Replace these dots with $k$'s and replace those $k$'s with dots.  Repeat this procedure for the numbers $k-1,k-2, \ldots, 1$, and then delete the dots.  For example, let $S$ be the increasing tableau on the far right in the above example and let $C$ be the set of outer corners marked with a dot.  The construction of $\rjdt_C(S)$ is now given by reading the sequence of tableaux in the above example from right to left, so that $\rjdt_C(S)$ is the increasing tableau with which we began.  

It is not difficult to show from the definitions that both $K$-theoretic jeu de taquin and its reverse yield increasing tableaux.   Moreover, if we start with an increasing tableau $T$ and a set of inner corners $B$ then 
$$T = \rjdt_C(S)$$
where $S = \fjdt_B(T)$ and $C$ is the set of outer corners given by the final location of the dots in the construction of $S$.

To prepare for the definition of the K-infusion operator, introduced in \cite{Thomas2007}, we define $\fjdt_T(U)$,  where $\lambda\subseteq\mu$, $T\in \Inc(\lambda)$, and $U \in \Inc(\mu/\lambda)$. Let $k$ be the largest entry in $T$ and let $B_i$ be the set of boxes in $T$ that contain the entry $i$.  
Setting $U_k: = U$ we recursively define 
$$U_{i-1} := \fjdt_{B_i}(U_{i})$$
for $i = k,k-1,\ldots, 1$.  (Note that $B_i$ is a collection of inner corners for $U_i$.)  Finally, we set 
$$\fjdt_T(U): = U_0.$$
For example, take the increasing tableaux

$$T =\begin{ytableau}
	1 & 2& 4\\
	2 & 3\\
	4
\end{ytableau}
\qquad 
U = \begin{ytableau}
{}	&&& 1 &2\\
	 & &1&3&4\\
	 &2&3
\end{ytableau}\ .$$
Then $U_3$ is constructed as
$$\begin{ytableau}
	\none&\none&\none[\bullet]& 1 &2\\
	\none &\none &1&3&4\\
	\none[\bullet] &2&3
\end{ytableau}
\qquad 
\begin{ytableau}
	\none&\none&1& \bullet &2\\
	\none &\none &\bullet&3&4\\
	\bullet &2&3
\end{ytableau}
\qquad 
\begin{ytableau}
	\none&\none&1& 2 &\bullet\\
	\none &\none &\bullet&3&4\\
	2 &\bullet&3
\end{ytableau}
\qquad 
\begin{ytableau}
	\none&\none&1& 2 &\bullet\\
	\none &\none &3&\bullet&4\\
	2 &3&\bullet
\end{ytableau}
\qquad 
\begin{ytableau}
	\none&\none&1& 2 &4\\
	\none &\none &3&4&\none[\bullet]\\
	2 &3&\none[\bullet]
\end{ytableau}$$
and $U_2$ is constructed as
$$\begin{ytableau}
	\none&\none&1& 2 &4\\
	\none &\none[\bullet] &3&4\\
	2 &3
\end{ytableau}
\qquad
\begin{ytableau}
	\none&\none&1& 2 &4\\
	\none & 3&\bullet&4\\
	2 &\bullet
\end{ytableau}\qquad \begin{ytableau}
	\none&\none&1& 2 &4\\
	\none & 3&4&\none[\bullet]\\
	2 &\none[\bullet]
\end{ytableau}$$
and $U_1$ is constructed as 
$$\begin{ytableau}
	\none&\none[\bullet]&1& 2 &4\\
	\none[\bullet] & 3&4\\
	2 
\end{ytableau}\qquad
\begin{ytableau}
	\none & 1&\bullet & 2 &4\\
	\bullet & 3&4\\
	2 
\end{ytableau}
\qquad
\begin{ytableau}
	\none & 1&2& \bullet &4\\
	2 & 3&4\\
	\bullet 
\end{ytableau}
\qquad
\begin{ytableau}
	\none & 1&2& 4 &\none[\bullet]\\
	2 & 3&4\\
	\none[\bullet] 
\end{ytableau}$$
and then an easy calculation shows that $\fjdt_T(U) = U_0$ is
$$\begin{ytableau}
	1&2& 4 \\
	2 & 3
\end{ytableau}\ .$$
As illustrated in this example, the fact that $T$ is a straight shape implies that $\fjdt_T(U)$ will also be a straight shape.  Using classical terminology, we call $\fjdt_T(U)$ a \emph{rectification} of $U$ by $T$.    In the classical setting  involving standard Young tableaux, rectification is invariant under the choice of a standard Young tableau $T$ with shape $\lambda$.  In the setting of increasing tableaux this is no longer the case.  In \cite{Gaetz2016}, Example 22, the authors give an example of a $U$ such that $\fjdt_T(U) \neq \fjdt_S(U)$ for two different choices of $T,S \in \Inc(\lambda)$.  

On the other hand, there are special cases where rectifications are unique. 
\begin{definition}
	An increasing tableau $T$ with straight shape is called a \emph{unique rectification target} (URT) if for every skew tableau $U$ that has $T$ as a rectification, $T$ is the only rectification of $U$. In this case we define $\Rect_T(\mu/\lambda)$ to be the set of all tableaux with shape $\mu/\lambda$ that rectify to $T$.    
\end{definition}
The phenomenon of URT's was first isolated in \cite{Buch2013}.  For our purposes we shall only need the following instance of this phenomenon, which was proved in \cite{Buch2013}, Corollary 4.7.
\begin{fact}\label{fact:min is URT}
Every minimal tableau is a unique rectification target.
\end{fact}

As expected, we can also define $\rjdt_U(T)$. To do so, let $k$ be the largest entry of $U$ and let $C_i$ be the set of boxes in $U$ that contain the entry $i$.  Set $T_1:= T$, and recursively define
$$T_{i+1} = \rjdt_{C_i}(T_i)$$
for $i = 1,2,\ldots, k$. Set $\rjdt_U(T):= T_{k+1}$.

For example, if we take $T$ and $U$ as above then similar calculations show that $T_2,T_3,T_4,T_5$ are, respectively, 
$$\begin{ytableau}
	\none & 1 & 2& 4\\
	2 & 3&4\\
	4
\end{ytableau}
\qquad
\begin{ytableau}
	\none & \none &1 &2 &4\\
	\none & 3&4\\
	2&4
\end{ytableau}
\qquad
\begin{ytableau}
	\none & \none &1 &2 &4\\
	\none & \none&3&4\\
	2&3&4
\end{ytableau}
\qquad
\begin{ytableau}
	\none & \none &\none &1 &2\\
	\none & \none&1&3&4\\
	2&3&4
\end{ytableau}\ .$$

From $\fjdt_T(U)$ and $\rjdt_U(T)$ we obtain the \emph{K-infusion} operator, which is defined as follows.
\begin{definition}
    For tableaux $T\in \Inc(\lambda)$ and $U\in \Inc(\mu/\lambda)$ we define
    $$\fus(T,U):= (\fjdt_T(U), \rjdt_U(T)).$$  
\end{definition}
Two fundamental properties of K-infusion, which we shall need, are the following.   

\begin{fact}\label{fact:infusion shape}
	Let $T\in \Inc(\lambda)$ and $U\in \Inc(\mu/\lambda)$ and set $(U',T') = \fus(T,U) $. Then there is some shape $\rho\subseteq \mu$ so that $U'\in \Inc(\rho)$ and $T'\in \Inc(\mu/ \rho)$.  
\end{fact}

\begin{fact}[{\cite[Theorem 3.1]{Thomas2007}}]\label{fact:infusion involution}
	Letting $T$ and $U$ be as above, we have 
	$$\fus(\fus(T,U)) = (T,U),$$
	i.e., the K-infusion operator is an involution.
\end{fact}

\begin{lemma}\label{lem:fus bij}
    Fix shapes $\tau \subseteq \nu$ and let $T\in \Inc(\tau)$ be a URT.  Then the function
    $$\fus(T,\cdot): \Inc(\nu / \tau) \to \bigcup \Big(\Inc(\lambda) \times \Rect_T(\nu/ \lambda)\Big),$$
    where the union is over all shapes $\lambda\subseteq \nu$, is a bijection.
\end{lemma}
\begin{proof}
It follows from the fact that K-infusion is an involution (Fact~\ref{fact:infusion involution}) that our mapping is injective.  To see that it is also surjective, consider some $(U,R)$ in the codomain and set
$$\fus(U,R) = (T,U').$$  
As $U \cup R$ has shape $\nu$, it follows from Fact~\ref{fact:infusion shape} that $T\cup U'$ also has shape $\nu$.  As $T$ has shape $\tau$, we see that $U'$ must have shape $\nu/\tau$. Fact~\ref{fact:infusion involution}  now implies that our mapping is surjective.  
\end{proof}

\begin{remark}\label{rmk:Inc}
We will later use the fact that Lemma~\ref{lem:fus bij} remains valid if we replace $Inc(\mu)$ throughout by $Inc^{st}(\mu)$.

\end{remark}

The special case of Lemma~\ref{lem:fus bij} where $\tau = \sigma_{n-1}, \nu = \sigma_n$ and $T = M_{n-1}$ will be used extensively in what follows.  In this case, we identify $\fW_n$ with  $\Inc(\sigma_n /\sigma_{n-1})$ by mapping the $i$th integer in a word to box $(n+1-i,i)$, using matrix notation, and our mapping becomes
$$\fus(M_{n-1},\cdot): \fW_n \to \bigcup \Big(\Inc(\lambda) \times \Rect_{M_{n-1}}(\sigma_n/ \lambda)\Big).$$

\begin{definition}

We define $\cW(w)$ and $\cM(w)$ by writing $\fus(M_{n-1},w)=(\cW(w), \cM(w))$.  

\end{definition}
As discussed in the introduction, one of the central points of this paper is to capture the RID and CID data of a word $w$ in a unified way.  We shall show that $\cW(w)$ captures RID data (see Theorem~\ref{thm:W implies rid}) and that $\cM(w)$ captures CID data (see Theorem~\ref{thm:M implies cid}). Our result for $\cW(w)$ will be easy to obtain, once we have a connection between K-jeu de taquin and K-Knuth equivalence.

Two increasing tableaux $T$ and $T'$ are said to be \emph{K-Knuth equivalent} if their reading words are K-Knuth equivalent. Recall that the \emph{reading word} of a tableau $T$ is the word $\Read(T)$ obtained by reading $T$ from bottom to top and left to right.  For example, if 
 $$T = \ytableaushort{22345,3447,667}$$
 then $\Read(T) = 667344722345$.  On the other hand, $T$ and $T'$ are said to be \emph{K-jeu de taquin equivalent}  if each can be obtained from the other by applying $\fjdt_B$ and $\rjdt_C$ for some finite number of choices of $B$ and $C$.
 
  \begin{fact}\label{fact:Kjdt and Knuth} (\cite{Buch2013}, Theorem 6.2)
 Increasing tableaux $T$ and $T'$ are K-Knuth equivalent if and only if they are K-jeu de taquin equivalent.
 \end{fact}

 \begin{thm}\label{thm:W implies rid}
 	If $\cW(u) = \cW(w)$ then $u \sim_{rid} w$.
 \end{thm}
  \begin{proof} 
  Since $\cW(u)$ and $\cW(w)$ are obtained from the tableaux representing $u$ and $w$ by applying $\fjdt$, it follows from Fact~\ref{fact:Kjdt and Knuth} that $u\equiv_{K} \Read(\cW(u))$ and $w\equiv_{K} \Read(\cW(w))$.  Since $\cW(u)=\cW(w)$, we obtain $u\equiv_K w$.
 
 By Fact~\ref{fact:Knuth on subintervals of rows}, this implies that for any integer interval $[a,b]$,  the words obtained from $u$ and $w$ by deleting all entries not in $[a,b]$ are $K$-Knuth equivalent. By Fact~\ref{fact:longest inc and dec},  this implies that the lengths of longest decreasing subsequences in these reduced words are the same, and the lengths of  longest increasing subsequences are the same. So $u\sim_{rid} w$.
\end{proof}
 
 \begin{thm}\label{thm:M implies cid}
 	If $\cM(u) = \cM(w)$ then $u \sim_{cid} w$.
 \end{thm}

The proof of Theorem~\ref{thm:M implies cid} requires new ideas and a significant amount of effort.  As mentioned in the introduction, its proof occupies the entirety of Section~\ref{sec:cM}. We close this section with the following remark.
 
 \begin{remark}
 \sloppy Let $S_{n-1}$ denote the superstandard tableau of shape $\sigma_{n-1}$, so that the first row of $S_{n-1}$ has entries $1,\ldots, n-1$, the second row has entries $n,\ldots,2n-3,$ and so on.  In~\cite{Thomas2011}, Theorem 4.2, Thomas and Yong consider (in our notation) $\fus(S_{n-1},w)$. They show that the first entry of $\fus(S_{n-1},w)$ is the Hecke insertion tableau of $w$.
 As can be seen from the proof of Theorem~\ref{thm:W implies rid}, that theorem would still hold with $M_{n-1}$ replaced by $S_{n-1}$; in fact it holds when $M_{n-1}$ is replaced by any increasing tableau of the same shape.  On the other hand, Theorem~\ref{thm:M implies cid} fails if we replace $M_{n-1}$ by $S_{n-1}$.  To illustrate, consider the words $u = 1 2 1 1 3$ and $w = 1 3 1 2 3$ which clearly do not have the same CID data (as can be seen from their last three letters). Yet, using $S_4$ in place of $M_4$ we find that $\cM(u) = \cM(w)$  is the tableau 
 $$\begin{ytableau}
	\none&\none&\none&3&4\\
	\none&2&6&7\\
	1&5&9\\
	5&8\\
	10
\end{ytableau}\ .$$
 \end{remark}

\section{Row shifting  and column shifting}\label{sec:row col shifting}
 \ytableausetup{mathmode, boxsize=1.05em}

 To motivate this section, we recall the definition of Sch\"{u}tzenberger's promotion operator in the context of standard Young tableaux \cite{Schutzenberger1963}. For such a tableau $T$, its promotion $\partial T$ is the tableau obtained by deleting the entry 1 from $T$ and decrementing the remaining entries by 1, then rectifying the resulting skew tableau (using jeu de taquin), and finally placing the largest entry of $T$ in the vacated box. 
 
 We mention without proof two results that follow from well-known properties of the Robinson-Schensted (RS) algorithm and Knuth equivalence. (We will prove an analogue of the second of these results in Theorem 3.1  below.) The first result states that if $\pi$ is a permutation of length $n$ and $\RS(\pi) = (P,Q)$ and $\pi' = \RS^{-1}(P,\partial Q)$
 then 
 \begin{equation}\label{eq:shift cid}
     \pi \sim_{rid} \pi'\qquad\myand\qquad \pi_2\ldots \pi_n\sim_{cid} \pi'_1\ldots \pi'_{n-1}.
 \end{equation}  
 We think of the mapping from $\pi$ to $\pi'$ as fixing RID data but ``shifting left" CID data.  
 
 Similarly, the second result is obtained by promoting the insertion tableau rather than the recording tableau.  To state this we introduce some notation in the broader context of words.  For any word $w$ and letter $a$, we let $w\setminus a$ be the word obtained by deleting all occurrences of $a$ from $w$.  We let $w^-$ denote the word obtained from $w$  by decrementing all the letters by 1.  With this notation, if we set $\pi' = \RS^{-1}(\partial P, Q)$ then 
 \begin{equation}\label{eq:shift rid}
     \pi \sim_{cid}\pi' \qquad\myand\qquad  (\pi\setminus 1)^- \sim_{rid} \pi'\setminus n.
 \end{equation}
 We think of \emph{this} mapping from $\pi$  to $\pi'$  as fixing CID data but ``shifting down" RID data.  
 
 Similar constructs have been employed in other contexts.  In particular, Rubey \cite{Rubey2011} used a map similar to \eqref{eq:shift cid}, in the context of semistandard Young tableaux, variants of the RSK algorithm, and growth diagrams, to study the distribution of strictly decreasing and weakly increasing (or vice versa) subsequences of words, with applications to moon polyominoes.  Likewise, Guo and Poznanovi\'c \cite{GuoPoznanovik2020} used a map similar to that in \eqref{eq:shift rid}, in the context of increasing tableaux and Hecke insertion, to prove their Conjecture 15 for the special case of stack polyominoes. Their methods did not provide a map with the properties of the map in \eqref{eq:shift cid}, however.  From our point of view, this is why they did not obtain a proof of the full conjecture.

 The goal of this section is to use the mapping $w\mapsto (\cW(w),\cM(w))$ to obtain analogues of the maps in \eqref{eq:shift cid} and \eqref{eq:shift rid} in the context of words. We denote these analogues by $\cC$ and $\cR$, respectively.  Our analogue $\cR$ of the map in  \eqref{eq:shift rid} will be obtained by using a map like that in \eqref{eq:shift rid}, although we will not use Hecke insertion or Hecke growth diagrams as Guo and Poznanovi\'{c} did. This map yields the special case of Conjecture 15 that Guo and Poznanovi\'{c} proved. Our second map, $\cC$, will be obtained by using new, and considerably more involved, methods. This map will be the key to obtaining a full proof of Conjecture 15.

 \subsection{Row shifting}
 To define our row-shifting operator we first describe the analogue of the classical promotion operator $\partial$ for elements of $Inc^{st}(\lambda)$, where $\lambda$ is a straight shape.   To start let $W\in \Inc^{st}(\lambda)$ and note that, by definition, the smallest entry in $W$ is 1 and occurs in box $(1,1)$, as $W$ is a straight shape.  Let $W^-$ be the tableau obtained by deleting $1$ from $W$ and decrementing all the remaining values by 1.  Finally, let $\partial W$ be the result of rectifying $W^-$, i.e., computing $\fjdt_{\{(1,1)\}} W^-$, and filling the vacated boxes with the largest letter in $w$. It is not hard to see that this gives a bijection 
$$\partial: \Inc^{st}(\lambda)\to \Inc^{st}(\lambda).$$
To simplify notation below, we establish the convention that for a pair of increasing tableaux $(A,B)$, we set $\partial(A,B): = (\partial A, B)$.

For each $n$ we let $\fW^{st}_n$ denote the set consisting of all elements of $\fW_n$ whose entries constitute an integer interval with left endpoint 1.  It follows from Lemma~\ref{lem:fus bij} and the remark following it that for each $n\geq 2$  we have a bijection
 
$$f:\fW^{st}_n \to \bigcup_\lambda \Big(\Inc^{st}(\lambda)\times \Rect_{M_{n-1}}(\sigma_n/\lambda)\Big)$$
given by $w\mapsto \fus(M_{n-1},w)$.
As a first step in defining our row-shift operator $\cR:\fW_n \to \fW_n$  we define 
$$\cR:\fW^{st}_n \to \fW^{st}_n$$
by letting $\cR(w) = (f^{-1}\circ \partial \circ f)(w)$.  When $n=1$, we define $\cR$ to be the identity map.  Note that, since $\partial$ is bijective on $Inc^{st}(\lambda)$, it follows that $\cR$ is bijective on $\fW_n^{st}$.

As an illustration, let $w = 4231142$.   The left-hand figure below depicts the pair $(M_6,w)$.  The next three figures then give $\fus(M_6,w) = (\cW(w),\cM(w))$, $(\partial \cW(w), \cM(w))$, and, finally, $\fus^{-1}(\partial \cW(w), \cM(w))$.  Hence $\cR(w) = 3121143$.  

$$
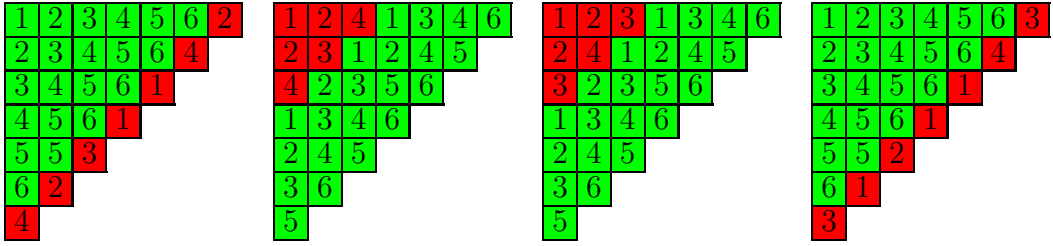
\begin{ytableau}
	*(green)1&*(green) 2 & *(green) 3 & *(green)4 & *(green)5& *(green) 6 & *(red) 2 \\
	*(green)2&*(green) 3 & *(green) 4 & *(green)5 & *(green)6& *(red) 4\\
	*(green)3&*(green) 4 & *(green) 5 & *(green)6 & *(red)1\\
	*(green)4&*(green) 5 & *(green) 6 & *(red)1\\
	*(green)5&*(green) 5 & *(red) 3\\
	*(green)6&*(red) 2\\
	*(red)4 
\end{ytableau}\quad 
\begin{ytableau}
	*(red)1&*(red) 2 & *(red) 4 & *(green)1 & *(green)3& *(green) 4 & *(green) 6 \\
	*(red)2&*(red) 3 & *(green) 1 & *(green)2 & *(green)4& *(green) 5\\
	*(red)4&*(green) 2 & *(green) 3 & *(green)5 & *(green)6\\
	*(green)1&*(green) 3 & *(green) 4 & *(green)6\\
	*(green)2&*(green) 4 & *(green) 5\\
	*(green)3&*(green) 6\\
	*(green)5 
\end{ytableau}\quad
\begin{ytableau}
	*(red)1&*(red) 2 & *(red) 3 & *(green)1 & *(green)3& *(green) 4 & *(green) 6 \\
	*(red)2&*(red) 4 & *(green) 1 & *(green)2 & *(green)4& *(green) 5\\
	*(red)3&*(green) 2 & *(green) 3 & *(green)5 & *(green)6\\
	*(green)1&*(green) 3 & *(green) 4 & *(green)6\\
	*(green)2&*(green) 4 & *(green) 5\\
	*(green)3&*(green) 6\\
	*(green)5 
\end{ytableau}\quad
\begin{ytableau}
	*(green)1&*(green) 2 & *(green) 3 & *(green)4 & *(green)5& *(green) 6 & *(red) 3 \\
	*(green)2&*(green) 3 & *(green) 4 & *(green)5 & *(green)6& *(red) 4\\
	*(green)3&*(green) 4 & *(green) 5 & *(green)6 & *(red)1\\
	*(green)4&*(green) 5 & *(green) 6 & *(red)1\\
	*(green)5&*(green) 5 & *(red) 2\\
	*(green)6&*(red) 1\\
	*(red)3 
\end{ytableau}$$ 

To extend the definition of $\cR$ to $\fW_n$, let $w\in \fW_n$ and suppose the entries of $w$ are $v_1< v_2<\cdots< v_k$. Let st($w$) be the word obtained from $w$ by replacing each $v_i$ by $i$, and define $\cR(w)$ to be the word obtained by taking $\cR$(st($w$)) and replacing each $i$ by $v_i$.

Our next theorem asserts that $\cal{R}$ has the desired properties.

 \begin{thm}\label{thm:row shift}
For all $n$, the mapping $\cR: \fW_n \to \fW_n$ is a bijection such that  $w\sim_{cid} \cR(w)$ and 
$$(w\setminus s)^- \sim_{rid} \cR(w)\setminus m,$$
where {s} and $m$ are the smallest and largest letters in $w$, respectively. Here, if $s< v_2<\cdots <v_k$ are the entries in $w$, then $(w\setminus s)^-$ denotes the word obtained from $w\setminus s$ by replacing $v_i$ by $v_{i-1}$ for each $2\leq i\leq k$. 

Finally, the set of integers occurring in $\cR(w)$ is the same as the set of integers occurring in $w$.
\end{thm}
\begin{proof} We need only consider $n\geq 2$.  By the definition of $\cR$ on $\fW_n$, it suffices to deal with the restriction of $\cR$ to $\fW_n^{st}$.

We have already shown that this restricted map is a bijection.  Set $u = \cR(w)$.  By our definition of $\cR$ we know $\cM(w) = \cM(u)$  and hence an application of Theorem~\ref{thm:M implies cid} tells us that $w \sim_{cid} u$.

We now turn our attention to the remaining claim, which, in the case under consideration, asserts that $$(w\setminus 1)^- \sim_{rid} \cR(w)\setminus m,$$ where
 $m$ is the largest letter in $w$.  We have 
$$\fus(M_{n-1},w)= (\cW(w),\cM(w)) \qquad \text{and}\qquad \fus(M_{n-1},u) = (\partial \cW(w),\cM(w)).$$  By Fact~\ref{fact:Kjdt and Knuth} we have
$$w\equiv_K \Read(\cW(w))\qquad \text{and}\qquad u \equiv_K \Read(\partial \cW(w)).$$
From the definition of $\partial$ and Fact~\ref{fact:Kjdt and Knuth} it follows that
$$(\Read(\cW(w))\setminus 1)^- \equiv_K \Read(\partial \cW(w)\setminus m).$$   Combining these equivalences and applying Fact~\ref{fact:Knuth on subintervals of rows} gives
$$(w\setminus 1)^- \equiv_K (\Read(\cW(w))\setminus 1)^- \equiv_K \Read(\partial \cW(w)\setminus m) \equiv_K u\setminus m.$$
An application of Facts~\ref{fact:Knuth on subintervals of rows} and \ref{fact:longest inc and dec} now shows that $(w\setminus 1)^- \sim_{rid} \cR(w)\setminus m.$

It is clear from our definition of $\cR$ that the set of integers in $\cR(w)$ is the same as the set of integers in $w$.
\end{proof}

\subsection{Column shifting}
\begin{thm}\label{thm:col shift}
	For every $n$, there exists an explicit bijection $\cC: \fW_n \to \fW_n$ such that if $u = \cC(w)$ then $w\sim_{rid} u$, 
	$$w_2\cdots w_n \sim_{cid} u_1\cdots u_{n-1},$$
	and the set of integers in $u$ is the same as the set of integers in $w$.
\end{thm}
 
 When $n=1,2$, we take $\cC$ to be the identity map.  For the remainder of this subsection assume $n\geq 3$.

To define our map $\cC$ we first make a few definitions.  Let $I_n\in \Inc(1^n)$ be the increasing tableau with entries $1,\ldots, n$.  Note that $I_n$ is a URT by Fact~\ref{fact:min is URT}. Next, define $N_n\in \Inc(\sigma_n)$ to be the increasing tableau obtained by concatenating $I_n$ with $(M_{n-1}+n)$, the result of incrementing all the entries of $M_{n-1}$ by $n$.   For example, $N_4$ is 
$$\begin{ytableau}
	*(yellow)1&*(green) 5 & *(green) 6 & *(green)7\\
	*(yellow)2&*(green) 6 & *(green) 7\\
	*(yellow)3&*(green) 7 \\
	*(yellow)4
\end{ytableau}\ .$$
We remark that since $I_n$ and $M_n$ are both URTs so too is $N_n$. 

To motivate our construction of $\cC$ we illustrate its calculation for the word $w = 4231142\in \fW_7$.  First we have
\begin{equation}\label{eq:map1}
    \begin{ytableau}
	*(yellow)1&*(green) 7 & *(green) 8 & *(green)9 & *(green)10& *(green) 11 & *(red) 2 \\
	*(yellow)2&*(green) 8 & *(green) 9 & *(green)10 & *(green)11& *(red) 4\\
	*(yellow)3&*(green) 9 & *(green) 10 & *(green)11 & *(red)1\\
	*(yellow)4&*(green) 10 & *(green) 11 & *(red)1\\
	*(yellow)5&*(green) 11 & *(red) 3\\
	*(yellow)6&*(red) 2\\
	*(red)4 
\end{ytableau} 
\quad\raisebox{-2em}{$\longrightarrow$}\quad
\begin{ytableau}
	*(red) 1 & *(red) 2 & *(red)4&*(yellow)1 & *(green)8 & *(green) 9 & *(green) 11 \\
	*(red) 2 & *(red) 3&*(yellow) 2 & *(green)7 & *(green)9 & *(green) 10\\
	*(red)4&*(green) 7 & *(green)8 & *(green)10 & *(green)11\\
	*(yellow)3&*(green) 8 & *(green) 9 & *(green)11\\
	*(yellow)4&*(green) 9 & *(green) 10\\
	*(yellow)5&*(green) 11\\
	*(yellow)6 
\end{ytableau}
\end{equation}
where the tableau on the left shows the pair $(N_6,w)$  and the tableau on the right  shows $\fus(N_6,w)$. 

Now we delete the ``vertical strip" consisting of all the yellow boxes and slide the green boxes down and left one unit.  Doing so will create a ``horizontal strip" of empty boxes that we color yellow and label from 1 to 6 as shown below on the left.  
\begin{align}\label{eq:map2}
\begin{ytableau}
	*(red) 1 & *(red) 2 & *(red)4&*(yellow) 3 & *(yellow)4 & *(yellow) 5 & *(yellow) 6 \\
	*(red) 2 & *(red) 3&*(yellow) 2 & *(green)8 & *(green)9 & *(green) 11\\
	*(red)4&*(yellow) 1 & *(green)7 & *(green)9 & *(green)10\\
	*(green)7&*(green) 8 & *(green) 10 & *(green)11\\
	*(green)8&*(green) 9 & *(green) 11\\
	*(green)9 &*(green)10\\
	*(green)11 
\end{ytableau}
\quad\raisebox{-2em}{$\longrightarrow$}\quad
\begin{ytableau}
	*(yellow) 1 & *(yellow) 2 & *(yellow)3 &*(yellow) 4 & *(yellow)5 & *(yellow) 6 & *(red) 2 \\
	*(green) 7 & *(green) 8&*(green) 9 & *(green)10 & *(green)11 & *(red) 3\\
	*(green)8 &*(green) 9 & *(green) 10 & *(green)11 & *(red)4\\
	*(green)9 &*(green) 10 & *(green) 11 & *(red)1\\
	*(green)10&*(green) 11 & *(red) 1\\
	*(green)1 &*(red) 3\\
	*(red)2 
\end{ytableau}
\end{align}
If we let $X$ denote the tableau of yellow and green boxes and $W$ denote the tableau of red boxes,  then the last figure shows $\fus(W,X)$.  Our desired word $\cC(w)$ is given by the red entries on the diagonal and so $\cC(w) = 2311432$.

Motivated by the map in \eqref{eq:map1}, we have the following lemma.

\begin{lemma}\label{lem:G_v}
    For $n\geq 3$, we have an explicit bijection 
    $$G_v: \fW_n \to \bigcup_{\mu,\lambda} \Big(\Inc(\lambda) \times \Rect_{I_{n-1}}(\mu / \lambda)\times \Rect_{M_{n-2}}(\sigma_n / \mu)\Big)$$ 
	where the union is over all partitions $\lambda\subseteq\mu\subseteq \sigma_{n}$ such that $\mu$ has length $n$.  Moreover, if $w\mapsto (W,V,M)$ then $w\equiv_K \Read(W)$ and  $M=\cM(w_2\cdots w_n)$.
\end{lemma}

\begin{proof}
By Lemma~\ref{lem:fus bij} we have the bijection
$$\fus(N_{n-1},\cdot): \fW_n \to \bigcup_{\lambda} \Big(\Inc(\lambda) \times \Rect_{N_{n-1}}(\sigma_n / \lambda)\Big),$$
where the union is over all $\lambda\subseteq \sigma_n$. Observe that any $T\in  \Rect_{N_{n-1}}(\sigma_n / \lambda)$ corresponds to a pair 
$(V,M)$,
where the entries $1,\ldots, n-1$ in $T$ determine $V$ and the remaining entries in $T$, decremented by $n$, determine $M$.  From the definitions involved we see that  $V \in \Rect_{I_{n-1}}(\mu/\lambda)$ and $M \in \Rect_{M_{n-1}}(\sigma_n/\mu)$ for some $\mu\supseteq\lambda$.  This mapping gives us the correspondence
$$\Rect_{N_{n-1}}(\sigma_n / \lambda) \to \bigcup_{\mu}\Big(\Rect_{I_{n-1}}(\mu / \lambda)\times \Rect_{M_{n-2}}(\sigma_n / \mu)\Big)$$
where the union is over all $\mu\subseteq \sigma_n$ of length $n$ containing $\lambda$. Note that $\mu$ must have length $n$ because, in the calculation of $\fus(N_{n-1},w)$, no entry from $M_{n-2}+n-1$ can move to the bottom box of $\sigma_n$, so this box must be filled by an entry from either $w$ or $I_{n-1}$, i.e., this box is in either $\lambda$ or $V=\mu/\lambda$, so it is in $\mu$.

To see that $G_v$ has the desired properties, note that, by our construction, the diagonal tableau representing $w$ is K-jdt equivalent to $W$, so by Fact~\ref{fact:Kjdt and Knuth}, $w\equiv_K \Read(W).$ Furthermore, our construction implies that if $\fus(M_{n-2}, w_2\cdots w_n) = (W',M')$ then $M'= M$ and so the last claim follows.
\end{proof}

We also have the following analogous result, where $*$ denotes conjugation.
\begin{lemma}\label{lem:G_h}
    For $n\geq 3$, we have an explicit bijection 
    $$G_h: \fW_n \to \bigcup_{\mu,\lambda} \Big(\Inc(\lambda) \times \Rect_{I^*_{n-1}}(\mu / \lambda)\times \Rect_{M_{n-2}}(\sigma_n / \mu)\Big)$$ 
	where the union is over all partitions $\lambda\subseteq\mu\subseteq \sigma_{n}$ such that $\mu$ has width $n$.  Moreover, if $w\mapsto (W,H,M)$ then $M= \cM(w_1\cdots w_{n-1})$, and  $w$ and $W$ are K-jdt equivalent.
\end{lemma}
\begin{proof}
First observe that conjugation gives a bijection between the codomain of $G_v$ and that of $G_h$.  Using this we define
$$G_h(w) = (W^*,V^*,M^*)$$
where $G_v(w^r) = (W,V,M)$ and $w^r$ is the reverse of $w$. It follows that $G_h$ is bijective and $\mu$ has width $n$.

We know $w^r$ and $W$ are K-jdt equivalent by Lemma~\ref{lem:G_v} and so $w$ and $W^*$ are K-jdt equivalent too. Additionally, since $(w^r_2\ldots w^r_n)^r = w_1\ldots w_{n-1}$
we see that 
$$\cM(w_1\ldots w_{n-1})  =\cM((w^r_2\ldots w^r_n)^r)= M^*.$$ 
\end{proof}

Our next definitions are inspired by the yellow boxes \eqref{eq:map2}.  

\begin{definitions}
    For partitions $\lambda \subseteq \mu\subseteq \sigma_n$ we say the skew shape $\mu/ \lambda$ is a \emph{vertical $n$-strip} provided that $\mu$ has length $n,$ $|\mu/ \lambda|$ is $n$ or $n-1,$ and $\mu_i-1\leq \lambda_i$ for all $i$. For such a vertical $n$-strip we define $\Vert_n(\mu/ \lambda)$ to be the set of increasing tableaux of shape $\mu/ \lambda$ with entries 1 through $n-1$ that weakly increase from top to bottom.  
    
    We say the skew shape $\mu/ \lambda$ is a \emph{horizontal $n$-strip} provided that  $\mu_1=n,$ $|\mu/  \lambda|$ is $n$ or $n-1,$ and $\mu_{i+1}\leq \lambda_i$ for all $i$. (The last condition says that no column can contain two boxes of $\mu/\lambda$.) For such a horizontal $n$-strip, we define $\Horiz_n(\mu/ \lambda)$ to be the set of increasing tableaux of shape $\mu/ \lambda$ with entries 1 through $n-1$ that weakly increase from left to right.
\end{definitions}

For example, if $\lambda = (3^2,2^2)$ and $\mu = (4^2, 3^2, 1^3)$ then $\Vert_n(\mu/ \lambda)$ consists precisely of the  two skew tableaux (shown in yellow):
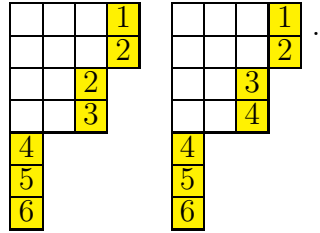
\ytableausetup{mathmode, boxsize=1em}
$$\begin{ytableau}
	{} &  & & *(yellow) 1\\
	&  & & *(yellow)2\\
	&  & *(yellow)2\\
	&  & *(yellow)3\\
	*(yellow)4\\
	*(yellow)5\\
	*(yellow)6
\end{ytableau}\quad 
\begin{ytableau}
	{}   &  & & *(yellow) 1\\
	 &  & & *(yellow)2\\
	 && *(yellow)3\\
	 && *(yellow)4\\
	*(yellow)4\\
	*(yellow)5\\
	*(yellow)6
\end{ytableau}\ .$$
On the other hand, if $\lambda = (3^2,2)$ and $\mu = (4^2, 2, 1^3)$ then $\Vert_n(\mu/ \lambda)$ consists of exactly one skew tableau shown below:
$$
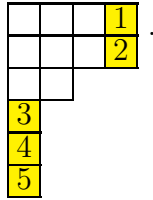
\begin{ytableau}
	{}   &  & & *(yellow) 1\\
	   &  & & *(yellow)2\\
	  &\\
	*(yellow)3\\
	*(yellow)4\\
	*(yellow)5
\end{ytableau}\ .$$
Observe that our motivating example illustrated a vertical $n$-strip with no repeated numbers. In general, at most one repetition can occur.

\begin{lemma}\label{lem:rect is vert}
    Fix $\lambda\subseteq \mu \subseteq \sigma_n$.  If $\mu$ has length $n$ we have $$\Rect_{I_{n-1}}(\mu / \lambda) = \Vert_n(\mu / \lambda)$$
and if $\mu$ has width $n$ we have 
    $$\Rect_{I^*_{n-1}}(\mu / \lambda) = \Horiz_n(\mu / \lambda).$$
\end{lemma}
\begin{proof}
Let $T \in \Vert_n(\mu / \lambda)$  and denote by
$$T_0,\ldots, T_m$$
the sequence of tableaux in a rectification of $T$ so that $T_0 = T$. (At this point we do not know $T_m = I_{n-1}$.) We will show by induction that each $T_i$ has at most one box in each row and has entries $1,\ldots,n-1$ that are weakly increasing from top to bottom.   This certainly holds for $T_0$.  Now consider adjacent boxes in $T_k$ like
$$	\begin{ytableau}
\bullet &a \\
b& \none
\end{ytableau}\ .$$
	 By our inductive hypothesis $a\leq b$ and so in $T_{k+1}$ these boxes become either
$$\begin{ytableau}
a &\bullet \\
b & \none
\end{ytableau}\qquad\text{ or }\qquad \begin{ytableau}
a &\bullet \\
\bullet & \none
\end{ytableau}$$	 
depending on whether $a<b$ or $a=b$.  As the set of entries in each $T_i$ is preserved it follows that $T_m$ has at most one box in each row and contains the entries $1,\ldots, n-1$.  As $T_m$ is a straight shape, it follows that $T_m = I_{n-1}$.  So $\Vert_n(\mu / \lambda) \subseteq \Rect_{I_{n-1}}(\mu / \lambda)$.

Next, fix $R\in \Rect_{I_{n-1}}(\mu/\lambda)$ and let 
$$R_0,\ldots, R_m$$
be the sequence of tableaux in a rectification of $R$ so that $R = R_0$ and $I_{n-1} = R_m$.   Since each step is invertible we can think of starting with $I_{n-1} = R_m$ and running this process backward.  An argument similar to the one above then shows that $R= R_0$ has at most one box in each row, has entries that weakly increase from top to bottom, and contains the numbers $1,\ldots, n-1$. It then follows that $R\in \Vert_n(\mu / \lambda)$ as needed.

The second assertion of the lemma follows from the first by taking conjugates.
\end{proof}

We now introduce a bijection between vertical $n$-strips and  horizontal $n$-strips that preserves $\lambda$.    For any shape $\mu\subseteq \sigma_n$ of length $n$ we define 
$$\tilde\mu  =( n, \mu_1-1, \ldots,\mu_n-1).$$
For such $\mu$ it follows that $\mu_n=1$ and therefore that $\tilde\mu$ has length at most $n$.

To get a feeling for how $\mu$ and $\tilde\mu$ relate consider the example where $n=7$, $\lambda=(3^2,2^2)$, and $\mu=(4^2,3^2,1^3)$ so that
 $\lambda\subseteq \mu\subseteq \sigma_7$. Coloring the boxes of $\lambda$ red, those of $\mu/ \lambda$ yellow, and those of $\sigma_7/\mu$ green gives
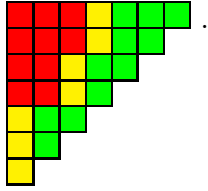
\ytableausetup{mathmode, boxsize=.8em}
$$\begin{ytableau}
	*(red)  &*(red)  &*(red) & *(yellow) & *(green)& *(green)& *(green)\\
	*(red)  &*(red)  &*(red) & *(yellow) & *(green)& *(green)\\
	*(red)  &*(red) & *(yellow) & *(green)& *(green)\\
	*(red)  &*(red) & *(yellow) & *(green)\\
	*(yellow) & *(green)& *(green)\\
	*(yellow) & *(green)\\
	*(yellow)
\end{ytableau}\ .$$
As in the motivating example, if we shift the green boxes down one unit and left one unit and then color all empty boxes yellow, we have the configuration
$$
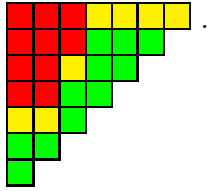
\begin{ytableau}
	*(red)  &*(red)  &*(red) & *(yellow) & *(yellow)& *(yellow)& *(yellow)\\
	*(red)  &*(red)  &*(red) & *(green) & *(green)& *(green)\\
	*(red)  &*(red) & *(yellow) & *(green)& *(green)\\
	*(red)  &*(red) & *(green) & *(green)\\
	*(yellow) & *(yellow)& *(green)\\
	*(green) & *(green)\\
	*(green) 
\end{ytableau}\ .$$
Note that $\tilde\mu$ is the shape of the yellow and red boxes.

\begin{lemma}\label{lem:vertical n-strip to horiz n-strip}

If $\mu/\lambda$ is a vertical $n$-strip, then $\tilde\mu/ \lambda$ is a horizontal $n$-strip so that $|\mu/\lambda| = |\tilde\mu/ \lambda|$.  Moreover, the skew shape $\sigma_n / \tilde\mu$ is obtained by shifting the shape $\sigma_n / \mu$ down one unit and left one unit.
\end{lemma}

\begin{proof}
Since $\mu_i-1\leq \lambda_i$ for all $1\leq i\leq n$,  it follows from our definition of $\tilde\mu$ that $\tilde\mu_{i+1}\leq \lambda_i$ for all $i$ and that $\tilde\mu_1=n$. To conclude the proof, it suffices to show that $\lambda\subseteq \tilde\mu,$ for then since $|\tilde\mu|=|\mu|$ we will have $|\tilde\mu/ \lambda|=|\mu/ \lambda|,$ which is $n$ or $n-1$. And so we aim to show that $\tilde\mu_i\geq \lambda_i$ for all $1\leq i\leq n$. 

For $i=1$ this is clear, since $\tilde\mu_1=n$ and $\lambda\subseteq \sigma_n$.  Now suppose $2\leq i\leq n$ so that $\tilde\mu_i=\mu_{i-1}-1$. If $\mu_{i-1}> \lambda_{i-1}$ then we have 
$$\tilde\mu_i = \mu_{i-1}-1\geq \lambda_{i-1}\geq \lambda_i.$$
Now consider the case when $\mu_{i-1}=\lambda_{i-1}$.  As  $\mu/ \lambda$ is a vertical $n$-strip with  $|\mu/ \lambda|\geq n-1$ we must have $
\mu_j=\lambda_j+1$ for all $j\neq i-1$. In particular, 
$$\lambda_i+1 = \mu_i \leq \mu_{i-1} = \lambda_{i-1}$$
and so 
$$\tilde\mu_i=\mu_{i-1}-1=\lambda_{i-1}-1\geq \lambda_i.$$

To address the last claim, note that after shifting  $\sigma_n / \mu$ down and to the left one unit its first row is empty and its $i$th row for $i>1$ extends, left to right, from column $\mu_{i-1}$ to column $n+1-i$. On the other hand, the first row of  $\sigma_n / \tilde\mu$ is empty (as $\tilde\mu_1 = n$) and its $i$th row for $i>1$ extends, left to right, from column $\tilde\mu_i +1 = \mu_{i-1}$ to column $n+1-i$.  (If $\mu_{i-1}> n+1-i$ then the $i$th row is empty in both cases.)
\end{proof}


We now get the following correspondence.  

\begin{lemma}\label{lem:Vert to Horiz}
	We have an explicit bijection $\varphi: \Vert_n(\mu/ \lambda) \to \Horiz_n(\tilde\mu/ \lambda)$.
\end{lemma}

\begin{proof}

	We consider two cases depending on $|\mu/ \lambda|$.

\bigskip
\noindent
	Case 1: $|\mu/ \lambda| = n-1$
\bigskip

	In this case the numbers 1 through $n-1$ each appear exactly once in each $V\in \Vert_n(\mu/\lambda)$.  This together with the fact that the entries of $V$ are weakly increasing from top to bottom implies that
	$$|\Vert_n(\mu/ \lambda)| = 1.$$ 
By the second claim in Lemma~\ref{lem:vertical n-strip to horiz n-strip} we see that $|\tilde\mu/ \lambda|= n-1$.  By an analogous argument $|\Vert_n(\tilde\mu/ \lambda)| =1$, making our bijection $\varphi$ a triviality in this case.   
	
\bigskip
\noindent	
	Case 2: $|\mu/ \lambda| = n$
\bigskip

We first claim that there is a natural correspondence between the outer corners of $\lambda$ and $\Vert_n(\mu/ \lambda)$. To see this observe that $\mu/\lambda$ has exactly one box on each of the $n$ rows, filled with the numbers 1 through $n-1$ so that they are weakly increasing from top to bottom.  Hence there must be exactly one repeated value and  every $V\in \Vert_n(\mu/ \lambda)$ is characterized by this repeated value.  On the other hand, for any outer corner $(i,j)$ of $\lambda$ we can construct $V$ by labeling its first $i$ boxes $1$ to  $i$ and then labeling its remaining $n-i$ boxes $i$ through $n-1$.  As $(i,j)$ is an outer corner of $\lambda$ the boxes labeled $i$ are not in the same column and so $V\in \Vert_n(\mu/\lambda)$. A straightforward check shows that this gives our desired correspondence.   

Again by Lemma~\ref{lem:vertical n-strip to horiz n-strip} we know that $|\tilde\mu/ \lambda| = n$.  By appealing to conjugates, we also obtain a correspondence between the outer corners of $\lambda$  and  $\Horiz_n(\tilde\mu/\lambda)$.  Consequently, we obtain a bijection
 $\varphi: \Vert_n(\mu/ \lambda) \to  \Horiz_n(\tilde\mu/ \lambda)$.  
\end{proof}

We are now ready to prove Theorem~\ref{thm:col shift}.  
\begin{proof}[Proof of Theorem~\ref{thm:col shift}]
For $n\geq 3$, we define our column-shift map $\cC:\fW_n\to \fW_n$ as follows.  For any $w\in \fW_n$ let $G_v(w) = (W,V,M)$ so that $V\in \Vert_{I_{n-1}}(\mu/\lambda)$ for some $\lambda\subseteq \mu$ by Lemma~\ref{lem:rect is vert}.   The same lemma together with Lemma~\ref{lem:Vert to Horiz}  shows that
$$\varphi(V)\in \Horiz(\tilde\mu/\lambda) = \Rect_{I^*_{n-1}}(\tilde\mu/\lambda).$$
Letting $M'$ be the tableau obtained by shifting $M$ down and left one unit, it follows from Lemma~\ref{lem:vertical n-strip to horiz n-strip} that $M'$ has shape $\sigma_n / \tilde\mu$.  We now see that 
$(W,\varphi(V),M')$ is in the codomain of $G_h$ and define
$$\cC(w) = G_h^{-1}(W,\varphi(V),M').$$
As $G_v, G_h$, and $\varphi$ are bijections, so is $\cC$. 

To check that $\cC$ has the desired properties set $u = \cC(w)$.  By our construction, both $w$ and $u$ are K-Knuth equivalent to $\Read(W)$, so $w\equiv_K u$ and by Fact~\ref{fact:longest inc and dec} we conclude that $w\sim_{rid} u$. Additionally,  $\cM(w_2\cdots w_n)$ and $\cM(u_1\cdots u_{n-1})$ are the same tableau (modulo a translation) and so  Theorem~\ref{thm:M implies cid} implies that 
$$w_2\cdots w_n \sim_{cid} u_1\cdots u_{n-1}.$$
The set of integers in $u$ is the same as the set of integers in $w$ because each of these sets is the set of integers in $W$.
\end{proof}

\section{$\cM(u)=\cM(w)$ implies $u\sim_{cid} w$}\label{sec:cM}

In this section we prove Theorem~\ref{thm:M implies cid}, thereby establishing that $\cM(w)$ determines CID data for words. To carry out this proof, we first establish an RSK-like algorithm, which we call iRSK.   Like classical RSK, iRSK gives a mapping from words to certain pairs of semistandard Young tableaux 
$$w\mapsto (\cP, \cQ)$$ 
with the same shape. This mapping is introduced in Subsection~\ref{subsec:irsk1}.  In Subsection~\ref{subsec:irsk2} we prove that if two words $u,w$ have the same $\cM$ then they have the same $\cQ$.   This requires a careful analysis of the rectification process of a word as defined in Subsection~\ref{subsec:k-infusion}.  Finally in Subsection~\ref{subsec:irsk3} we prove that $\cQ$ generalizes the classical recording tableau in that $Q(u) = Q(w)$ implies $u\sim_{cd} w$.  Combining these results, together with an appeal to word reversal, gives us a proof that if $\cM(u)=\cM(w)$ then $u\sim_{cid} w$.

\subsection{An inflated RSK algorithm}\label{subsec:irsk1} 
 \ytableausetup{mathmode, boxsize=1.1em}
 
 In this subsection we develop an insertion algorithm for words, called $\iRSK$, that is similar to classical RSK.  Throughout this subsection we make use of the RSK algorithm and the surrounding language without giving a full review.  A treatment of these ideas can be found in~\cite{Fulton1996,Sagan2001,Stanley2012}.

Recall that the classical RSK algorithm consists of two parts: ``insertion" and ``bumping".    Our first definition isolates the former.
 
 \begin{definition}
     For any semistandard Young tableau $p$ consisting of a single row and any  positive integer $a$ let
     $$p\dashleftarrow a$$
     be the semistandard tableau $p'$ obtained as follows.  If $p$ is empty or $a$ is greater than or equal to all the entries of $p$ then $p'$ is the result of appending $a$ to the (right) end of $p$.  Otherwise, $p'$ is the result of replacing the leftmost entry of $p$ that is larger than $a$ by $a$.

     If $w\in \fW_n$ then we define $p\dashleftarrow w$ as 
     $$ ((p \dashleftarrow w_1) \dashleftarrow w_2)\cdots )\dashleftarrow w_n.$$
 \end{definition}
 For example, 
 $$\left(\ytableaushort{224556}\dashleftarrow 351\right) = \ytableaushort{123555}\ .$$

To motivate our variant of the insertion process, which we will call i-insertion, consider the tableau
$$P' = \ytableaushort{224556,446}\ .$$ 
To i-insert a 5 into this tableau we first replace $P'$ by the skew tableau
$$\begin{ytableau}
\none&\none&\none&\none&\none&\none&5\\
2&2&4&5&5&6\\
4&4&6
\end{ytableau}$$
and denote the rows by $p_0', p_1',p_3'$.  We inflate $w = 5$ to $w^+ = 55$ so that the rightmost occurrence of 5 in $$p_1 : = \left(p_1' \dashleftarrow 55\right) = \ytableaushort{2245555}$$
is in the same column as the leftmost occurrence of $5$ in row $p_0'$.  Note that the set of elements bumped from $p_1'$  forms the word $\alpha = 6$. We inflate this word to $666$ so that in 
$$p_2 : = \left(p_2' \dashleftarrow \ytableaushort{666}\right) = \ytableaushort{446666}$$
the rightmost $6$ is in the same column as the leftmost 6 in $p_1'$.  In this step nothing was bumped so we stop.    The resulting tableau from this operation is then 
$$\ytableaushort{2245555,446666}\ .$$
We formalize this insertion algorithm in the following definition.  

 \begin{definition}
Given a semistandard Young tableau $P'$, denote its $i$th row by $p_i'$.  For any positive integer $a$ we define $P'\twoheadleftarrow a$ to be the semistandard Young tableau $P$ with rows $p_i$ determined as follows.   

First redefine $P'$ as $P'\oplus \raisebox{-.3em}{\ytableaushort{a}}$, index the rows from 0, and set $S_0 = \{a\}$.  We now proceed recursively.  Given the set $S_i = \{a_1<\cdots <a_k\}$, whose values are a subset of $p_{i}'$, let $\alpha_i^+$ be the inflation of the word $\alpha=a_1\cdots a_k$ so that if
$$p_{i+1} := p_{i+1}'\dashleftarrow \alpha_i^+$$
then the rightmost occurrence of $a_j$ in $p_{i+1}$ is in the same column as the leftmost occurrence of $a_j$ in $p_i'$.   Finally, let $S_{i+1}$ be the set of elements bumped during this insertion process.  

For any word $w$ we define its \emph{i-insertion tableau} $\cP(w)$ to be the tableau
$$((\emptyset\twoheadleftarrow w_1 )\twoheadleftarrow w_2 )\cdots )\twoheadleftarrow w_n.$$
We define the \emph{i-recording tableau} $\cQ(w)$ as follows.  For $j\geq 1$, let $\lambda^{(j)}$ be the shape of $\cP(w_1\cdots w_j)$ and set $\lambda^{(0)} = \emptyset$ so that 
$$\lambda^{(0)} \leq \cdots \leq \lambda^{(n)}.$$
Let $\cQ(w)$ be the tableau with shape $\lambda^{(n)}$ such that its entries in $\lambda^{(j)} / \lambda^{(j-1)}$ are $j$ for $1\leq j\leq n$.  Finally we write 
$$\iRSK(w) = (\cP(w), \cQ(w)).$$ 
 \end{definition}

Let us (again) illustrate i-insertion but this time using the notation from our definition.  To start take $a =1$ and 
$$P' = \ytableaushort{2255566,347788}\ .$$
Then to compute $P' \twoheadleftarrow 1$ we first redefine $P'$ as 
$$\begin{ytableau}
\none & \none & \none & \none & \none & \none &\none &  1\\
2&2&5&5&5&6&6\\
3&4&7&7&8&8
\end{ytableau}$$
and set $S_0 =\{1\}$. Now $\alpha_0^+ = 1^8$ so that 
$$p_1=\left(\underbrace{\raisebox{-.2em}{\ytableaushort{2255566}}}_{p'_1}\dashleftarrow \alpha_0^+\right) = \raisebox{-.2em}{\ytableaushort{11111111}}$$
and $S_1=\{2,5,6\}$.  Now $\alpha_1^+ = 25666$ so that 
$$p_2=\left(\underbrace{\raisebox{-.2em}{\ytableaushort{347788}}}_{p'_2}\dashleftarrow \alpha_1^+\right) = \raisebox{-.2em}{\ytableaushort{245666}}$$
and $S_2=\{3,7,8\}$.  Now $\alpha_2^+ = 37788$ so that
$$p_3=\left(\emptyset\dashleftarrow \alpha_2^+\right) = \raisebox{-.2em}{\ytableaushort{37788}}$$
and $S_3=\emptyset$.  
Consequently, 
$$P' \twoheadleftarrow 1 = \ytableaushort{11111111,245666,37788}\ .$$

\bigskip
To illustrate the entire $\iRSK$ algorithm, take $w = 4163276$.  Then we have
\renewcommand{\arraystretch}{2.5}
$$\begin{array}{c|ll}
w_1\cdots w_j & \cP(w_1\cdots w_j) & \cQ(w_1\cdots w_j)\\
\hline
	4& \ytableaushort{4} & \ytableaushort{1}\\
	41 & \ytableaushort{11,4} & \ytableaushort{12,2} \\
	416 & \ytableaushort{116,4} & \ytableaushort{123,2}\\
	4163 & \ytableaushort{1133,466}& \ytableaushort{1234,244}\\
	41632 & \ytableaushort{11222,333,46}  & \ytableaushort{12345,244,55}\\
	416327 & \ytableaushort{112227,333,46} & \ytableaushort{123456,244,55}\\
	4163276 & \ytableaushort{1122266,333777,46} & \ytableaushort{1234567,244777,55}\ .
\end{array}$$

We leave it to the reader to show that in general $\cP(w)$ and $\cQ(w)$ are semistandard Young tableaux. 

We can of course extend the definition of iRSK from words to 2-line arrays, just as in RSK (see Subsection 4.3).  We shall not need this extension in the present paper.

 \subsection{$\cM(u)=\cM(w)$ implies $\cQ(u)=\cQ(w)$}\label{subsec:irsk2}

The goal of this subsection is to prove that if two words $u,w$ have the same $\cM$ then they have the same i-recording tableau $\cQ$. To carry this out we first introduce the following notation.
\begin{notation}
Let $\W{}{0}(w)$ be the diagonal skew-tableau that represents the word $w$. (See the paragraph following Remark~\ref{rmk:Inc}.)  For $j\geq 1$, let $\W{}{j}(w)$ be the tableau obtained from $\W{}{0}(w)$ by carrying out the first $j$ steps in the calculation of $\cW(w) = \fjdt_{M_{n-1}}(\W{}{0}(w))$, i.e., 
$$\fjdt_{B_{n-j}}\circ\cdots \circ \fjdt_{B_{n-1}}(\W{}{0}(w)),$$
where $B_i$ is the set of boxes in $M_{n-1}$ containing $i$.

Let $\kW{}{j}{0}(w)$ be the tableau consisting of $\W{}{j}(w)$ together with the placement of dots used to compute $\W{}{j+1}(w)$, and for $k\geq 1$ let $\kW{}{j}{k}(w)$ be the tableau (including boxes with dots) that results from interchanging the dots with $1$, then $2,\ldots$, up through $k$.
\end{notation}

For example if $w = 4351243$ then 
$$\W{}{0}(w) = \hspace*{-1cm}\begin{ytableau}
\none&\none&\none&\none&\none&\none&3\\
\none&\none&\none&\none& \none &4\\
\none&\none&\none&\none&2\\
\none&\none&\none&1\\
\none&\none&5 \\
\none&3 \\
4\\
\end{ytableau}
\qquad
\W{}{1}(w) =\hspace*{-1cm} \begin{ytableau}
\none&\none&\none&\none&\none&3\\
\none&\none&\none&\none& 2 &4\\
\none&\none&\none&1\\
\none&\none&1\\
\none&3&5 \\
3 \\
4\\
\end{ytableau}$$
and then 
$$\W{}{1,0}(w) = \hspace{-1cm}\begin{ytableau}
\none&\none&\none&\none&\none[\bullet]&3\\
\none&\none&\none&\none[\bullet]& 2 &4\\
\none&\none&\none[\bullet]&1\\
\none&\none[\bullet]&1\\
\none[\bullet]&3&5 \\
3 \\
4\\
\end{ytableau}
\qquad
\W{}{1,1}(w) = \hspace*{-1cm}\begin{ytableau}
\none&\none&\none&\none&\bullet&3\\
\none&\none&\none&1& 2 &4\\
\none&\none&1&\bullet\\
\none&1&\bullet\\
\bullet&3&5 \\
3 \\
4\\
\end{ytableau}
\qquad
\W{}{1,2}(w) = \hspace*{-1cm}\begin{ytableau}
\none&\none&\none&\none&2&3\\
\none&\none&\none&1& \bullet &4\\
\none&\none&1&\bullet\\
\none&1&\bullet\\
\bullet&3&5 \\
3 \\
4\\
\end{ytableau}\ .$$

Using this notation we now give an outline of this section.  Recall that $\cQ(w)$ is completely determined by the shapes of $\cP(w_1\cdots w_j)$ for all $j$.  Our strategy is to show that, for all $j$, $\cM(w)$ determines (part of) the shape of $\W{}{j}$ (Lemma~\ref{lem:cM determines cW}) which in turn determines the shape of $\cP(w_1\cdots w_j)$ (Lemma~\ref{lem:P = W}).  

To facilitate our arguments, observe first that for any word $w$ the shape of $\W{}{j}(w)$ is such that the first elements in each of the first $n-j$ columns are along the diagonal given by the boxes $(n-j,1), (n-j-1,2), \cdots,  (1,n-j)$.  It turns out that in what follows, we need only concern ourselves with these $n-j$ initial columns.  Since they start on successive rows we introduce the following (unorthodox) coordinate system to reference boxes among these $n-j$ initial columns. In $\W{}{j}(w)$ we write $\Wcoor{k}{i}$ to indicate the $k$th box (from the top) in column $i\leq n-j$.  For convenience we further define $\Wcoor{0}{i}$ when $i<n-j$ to be the box that is directly above $\Wcoor{1}{i}$.

For example, consider the tableau
$$\W{}{4}(4351243)= \begin{ytableau}
\none &\none & 1 & 2&3\\
\none &1 & 4\\
1 & 5 \\
3 \\
4\\
\end{ytableau}\ .$$
 The first $7-4 =3$ columns have topmost boxes that are along a diagonal. Under our new coordinate system, $\Wcoor{2}{3}$ references the box containing the 4 in the third column.  Likewise, $\Wcoor{2}{1}$ references the box containing the 3 in the first column.

It is important to note that under this coordinate system $\Wcoor{k}{i}$ in $\W{}{j}(w)$ references the same box as $\Wcoor{k+1}{i}$ in $\W{}{j+1}(w)$ whenever $i+j< n$.   Furthermore, when working with the skew tableaux  $\kW{}{j}{k}(w)$ we shall use the coordinate system for $\W{}{j}(w)$. For example, in 
$$\kW{}{4}{2}(4351243)= \begin{ytableau}
\none &1 & 2 & \bullet&3\\
1 &\bullet & 4\\
\bullet & 5 \\
3 \\
4\\
\end{ytableau}$$
the 1's are in positions $\Wcoor{0}{1}$ and $\Wcoor{0}{2}$ while $\Wcoor{1}{3}$ references the box containing the 2.

To motivate our first few lemmas, consider the intermediate tableaux in the construction of $\cP(w)$  where $w = 4351243$:
$$\ytableaushort{4}\quad \ytableaushort{33,4}\quad \ytableaushort{335,4}\quad \ytableaushort{1111,355,4}\quad \ytableaushort{11112,355,4}\quad \ytableaushort{111124,355,4}\quad \ytableaushort{1111233,344444,45}\ .$$
Likewise consider the tableaux $\W{}{0}(w),\ldots,\W{}{6}(w)$ in the construction of $\cW(w)$:
 \ytableausetup{mathmode, boxsize=1em}
$$\begin{ytableau}
\none &\none &\none &\none &\none &\none &3\\
\none &\none &\none &\none &\none &4\\
\none &\none &\none &\none &2\\
\none &\none &\none &1\\
\none &\none &5\\
\none &3\\
4\\
\end{ytableau}\ \begin{ytableau}
\none &\none &\none &\none &\none & 3\\
\none &\none &\none &\none & 2 & 4\\
\none &\none &\none & 1\\
\none &\none &1\\
\none &3 & 5\\
3 \\
4\\
\end{ytableau}\ 
\begin{ytableau}
\none &\none &\none &\none &2 & 3\\
\none &\none &\none &1 & 4\\
\none &\none & 1\\
\none &1 & 5\\
3 & 5\\
4\\
\end{ytableau}\quad
\begin{ytableau}
\none &\none &\none &1 &2 & 3\\
\none &\none &1 & 4\\
\none &1 & 5\\
1 & 5 \\
3 \\
4\\
\end{ytableau}\quad
\begin{ytableau}
\none &\none & 1 & 2&3\\
\none &1 & 4\\
1 & 5 \\
3 \\
4\\
\end{ytableau}$$

$$
\begin{ytableau}
\none &1 & 2& 3\\
1 & 4 \\
3 & 5\\
4\\
\end{ytableau}\quad \begin{ytableau}
1 & 2& 3\\
3& 4 \\
4 & 5
\end{ytableau}\ .$$

\begin{notation}
For any skew tableau $T$ we denote its $i$th column by $T_i$.
\end{notation}

Observe that the first column of $\cP(w)$ is the same as the first column in $\W{}{6}(w)$, and the second column of $\cP(w)$ is the same as the second column in $\W{}{5}(w)$, and in general 
$$\W{i}{7-i}(w)=\cP_i(w_1\cdots w_7).$$
Additionally, we see that the first column of $\cP(w_1\cdots w_6)$ is the same as the first column in $\W{}{5}(w)$, and the second column is the same as the second column in $\W{}{4}(w)$, etc.  In general we have 
$$\W{i}{6-i}(w)=\cP_i(w_1\cdots w_6).$$
We first work to prove that this pattern holds in general.

\begin{lemma}\label{lem:cols of W}
	For any word $w$ of length $n$ and $m\leq n$ we have $\W{i}{j}(w) = \W{i}{j}(w_1\cdots w_m)$ for all $i+j\leq m$.  
\end{lemma}
\begin{proof}
We begin with a general observation.  Assume $u$ and $w$ are two words such that for some $j$ and some $k$ we have
$$\W{i}{j-1}(w) = \W{i}{j-1}(u),$$
for all $i\leq k+1$.   By definition, it follows that
$$\W{i}{j}(w) = \W{i}{j}(u),$$
for all $i\leq k$.  

Now observe that if $j =0$ then  
$$\W{i}{0}(w) = w_i = \W{i}{0}(w_1\cdots w_m),$$ 
for all  $i \leq m$.   By the above observation, it follows that
$$\W{i}{1}(w) = w_i = \W{i}{1}(w_1\cdots w_m),$$ 
for all $i+1\leq m$.  Repeating this argument proves our claim.
\end{proof}

\begin{lemma}\label{lem:W,P}
    Let $w$ be a word of length $n$ and fix $i\geq 1$.  If
\begin{align}
  \W{i-1}{n-i}(w) &= \cP_{i-1}(w_1\cdots w_{n-1})\label{eq:W,P 1}\\
  \W{i}{n-i-1}(w) &= \cP_i(w_1\cdots w_{n-1})\label{eq:W,P 2}\\
\W{i+1}{n-i-1}(w) &= \cP_{i+1}(w)\label{eq:W,P 3}
\end{align}    
then $\W{i}{n-i}(w) = \cP_i(w)$.
\end{lemma}
\begin{proof}

To see that the first entries of $\cP_i(w)$ and $\W{i}{n-i}(w)$ are the same, define $s_i=\min\{w_i,\ldots, w_n\}$ for $1\leq i\leq n$. By the definition of $\W{}{n-i}(w),$ the first entry of $\W{i}{n-i}(w)$ is $s_i$.

Now for $1\leq i\leq n$ define $t_i$ as follows.  Let $a_1$ be the smallest entry in $w$ and let the rightmost occurrence of $a_1$ be in position $r_1$. For $1\leq i\leq r_1$, let $t_i=a_1$.  Then let $a_2$ be the smallest entry occurring to the right of position $r_1$ in $w$ and let the rightmost occurrence of $a_2$ in $w$ be in position $r_2$.  For $r_1+1\leq i\leq r_2$, let $t_i=a_2$. Continue in this way.  By the definition of $i$-insertion, the first entry of $\cP_i(w)$ is $t_i$. It is easy to see, by thinking about the way the $s_i$ are chosen, that $s_i=t_i$ for $1\leq i\leq n$.

Having shown that the first entries of $\cP_i(w)$ and $\W{i}{n-i}(w)$ are equal, we now assume that their $(k-1)$st entries are equal with the aim of proving that their $k$th entries also coincide.  To this end we consider two cases depending on whether there is ever a dot in position $\langle k-1,i\rangle$ during the transformation from $\W{}{n-i-1}(w)$ to $\W{}{n-i}(w)$.  

\bigskip
\noindent
\textbf{Case 1:} The position $\Wcoor{k-1}{i}$ never contains a dot.
\bigskip

Assume for the moment that $i\geq 2$ and let the following diagrams depict $\W{}{n-i-1}(w)$ and $\W{}{n-i}(w)$, respectively, where $a$ is the $(k-1)$st element in the $i$th column on the left.  By the assumption of Case 1, $a$ is the $k$th element in the $i$th column on the right.  (Here each letter denotes a positive integer or $\infty$, indicating an empty box. It is convenient here to use $\infty$ rather than 0.)  

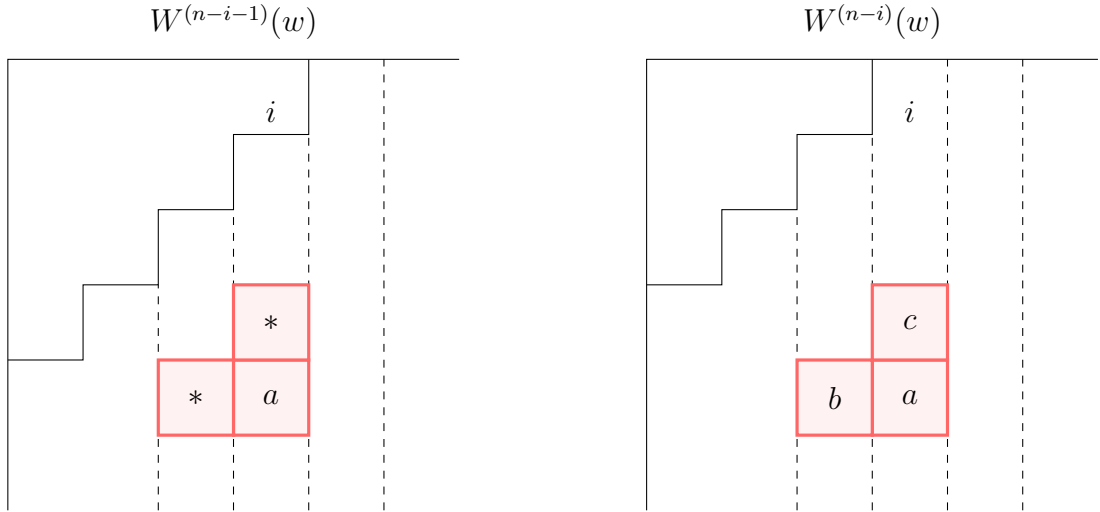
\begin{center}
    \begin{tikzpicture}[squarednode/.style={rectangle, draw=red!60, fill=red!5, very thick, minimum size=10mm}] 
        \draw(0,0) -- (0,6) -- (6,6);
        
        \draw(0,2) -- (1,2) --(1,3) -- (2,3) -- (2,4) -- (3,4) -- (3,5) -- (4,5) -- (4,6) ;
        
        \draw[dashed](2,0) -- (2,3);
        \draw[dashed](3,0) -- (3,4);
        \draw[dashed](4,0) -- (4,5);
        \draw[dashed](5,0) -- (5,6);

        \node[squarednode] at (3.5,1.5) {$a$};
        \node[squarednode] at (2.5,1.5) {$*$};
        \node[squarednode] at (3.5,2.5) {$*$};
        \node at (3,6.5) {$W^{(n-i-1)}(w)$};
        
        \node at (3.5,5.3) {$i$};
    \end{tikzpicture}\qquad\qquad\qquad
     \begin{tikzpicture}[squarednode/.style={rectangle, draw=red!60, fill=red!5, very thick, minimum size=10mm}] 
        \draw(0,0) -- (0,6) -- (6,6);
        
        \draw(0,3) -- (1,3) --(1,4) -- (2,4) -- (2,5) -- (3,5)--(3,6);
        
        \draw[dashed](2,0) -- (2,4);
        \draw[dashed](3,0) -- (3,5);
        \draw[dashed](4,0) -- (4,6);
        \draw[dashed](5,0) -- (5,6);

        \node[squarednode] at (3.5,1.5) {$a$};
        \node[squarednode] at (2.5,1.5) {$b$};
        \node[squarednode] at (3.5,2.5) {$c$};
        \node at (3,6.5) {$W^{(n-i)}(w)$};
        \node at (3.5,5.3) {$i$};
    \end{tikzpicture}
\end{center}

Note that \eqref{eq:W,P 1} implies that $b$ is in position $(k-1,i-1)$ in $\cP(w_1,\ldots, w_{n-1})$.  Additionally, \eqref{eq:W,P 2} implies that $a$ is in position $(k-1,i)$ in $\cP(w_1\cdots w_{n-1})$. Finally, by our inductive hypothesis we know that $c$ is in position $(k-1,i)$ in $\cP(w)$.  

Now assume $a<\infty$.  As $c,b<a$ it follows  by the definition of i-insertion that $a$ must be in position $(k,i)$ in $\cP(w),$ as required.

If $a=\infty$ then the above argument implies, by the definition of i-insertion, that the position $(k,i)$ in $\cP(w)$ must be empty. 

Finally, if $i=1$ then $a$ is in the first column of $\cP(w_1\cdots w_{n-1})$ and the same argument, in the absence of $b$, yields the desired result.  

\bigskip
\noindent
\textbf{Case 2:} The position $\Wcoor{k-1}{i}$ eventually contains a dot.
\bigskip

Assume for the moment that $i\geq 2$ and let the following diagrams depict $\W{}{n-i-1}(w)$ and $\W{}{n-i}(w)$ respectively, where $a$ is the $k$th element in the $i$th column on the left and $b'$ is the $k$th element in the $i$th column on the right.

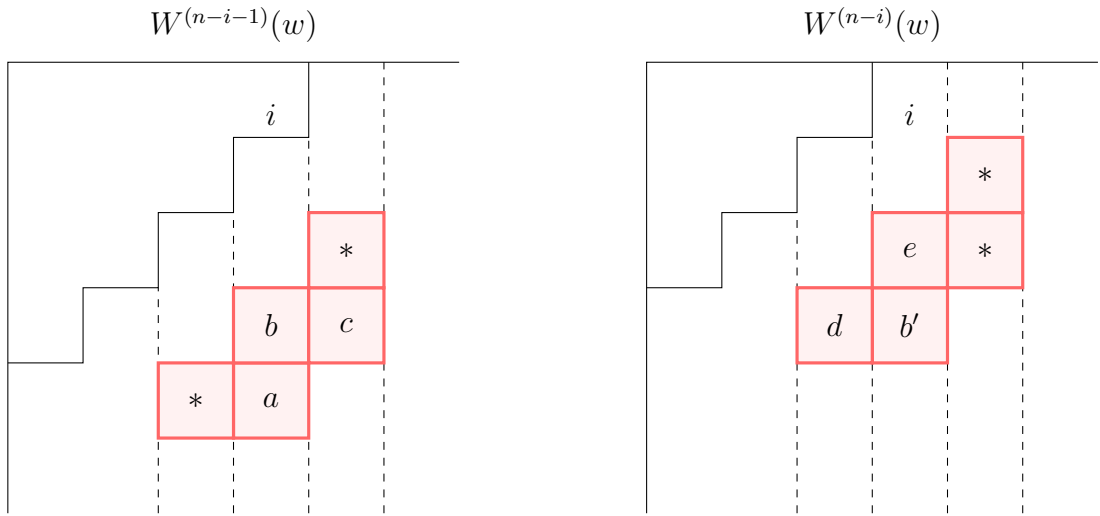
\begin{center}
    \begin{tikzpicture}[squarednode/.style={rectangle, draw=red!60, fill=red!5, very thick, minimum size=10mm}] 
        \draw(0,0) -- (0,6) -- (6,6);
        
        \draw(0,2) -- (1,2) --(1,3) -- (2,3) -- (2,4) -- (3,4) -- (3,5) -- (4,5) -- (4,6) ;
        
        \draw[dashed](2,0) -- (2,3);
        \draw[dashed](3,0) -- (3,4);
        \draw[dashed](4,0) -- (4,5);
        \draw[dashed](5,0) -- (5,6);

        \node[squarednode] at (3.5,1.5) {$a$};
        \node[squarednode] at (2.5,1.5) {$*$};
        \node[squarednode] at (3.5,2.5) {$b$};
        \node[squarednode] at (4.5,2.5) {$c$};
        \node[squarednode] at (4.5,3.5) {$*$};
        \node at (3,6.5) {$W^{(n-i-1)}(w)$};
        \node at (3.5,5.3) {$i$};
    \end{tikzpicture}\qquad\qquad\qquad
     \begin{tikzpicture}[squarednode/.style={rectangle, draw=red!60, fill=red!5, very thick, minimum size=10mm}] 
        \draw(0,0) -- (0,6) -- (6,6);
        
        \draw(0,3) -- (1,3) --(1,4) -- (2,4) -- (2,5) -- (3,5)--(3,6);
        
        \draw[dashed](2,0) -- (2,4);
        \draw[dashed](3,0) -- (3,5);
        \draw[dashed](4,0) -- (4,6);
        \draw[dashed](5,0) -- (5,6);

        \node[squarednode] at (3.5,2.5) {$b'$};
        \node[squarednode] at (2.5,2.5) {$d$};
        \node[squarednode] at (3.5,3.5) {$e$};
        \node[squarednode] at (4.5,3.5) {$*$};
        \node[squarednode] at (4.5,4.5) {$*$};
        \node at (3,6.5) {$W^{(n-i)}(w)$};
        \node at (3.5,5.3) {$i$};
    \end{tikzpicture}
\end{center}

It follows from our assumptions that in $\cP(w_1\cdots w_{n-1})$ and $\cP(w)$ the $i-1, i, i+1$ columns and $k-1, k$ rows give, respectively, the rectangles

\begin{center}
    \begin{tikzpicture}[squarednode/.style={rectangle, draw=red!60, fill=red!5, very thick, minimum size=10mm}] 
        \draw(0,0) -- (0,4) -- (6,4);
        
        \draw[dashed](2,0) -- (2,4);
        \draw[dashed](3,0) -- (3,4);
        \draw[dashed](4,0) -- (4,4);
        \draw[dashed](5,0) -- (5,4);
        
        \draw[dashed](0,1) -- (6,1);
        \draw[dashed](0,2) -- (6,2);
        \draw[dashed](0,3) -- (6,3);

        \node[squarednode] at (2.5,1.5) {$*$};
        \node[squarednode] at (2.5,2.5) {$d$};
        
        \node[squarednode] at (3.5,1.5) {$a$};
        \node[squarednode] at (3.5,2.5) {$b$};
        
        \node[squarednode] at (4.5,2.5) {$*$};
        \node[squarednode] at (4.5,1.5) {$*$};
        \node at (3.5,4.3) {$i$};
    \end{tikzpicture}\qquad\qquad\qquad
     \begin{tikzpicture}[squarednode/.style={rectangle, draw=red!60, fill=red!5, very thick, minimum size=10mm}] 
        \draw(0,0) -- (0,4) -- (6,4);
        
        \draw[dashed](2,0) -- (2,4);
        \draw[dashed](3,0) -- (3,4);
        \draw[dashed](4,0) -- (4,4);
        \draw[dashed](5,0) -- (5,4);
        
        \draw[dashed](0,1) -- (6,1);
        \draw[dashed](0,2) -- (6,2);
        \draw[dashed](0,3) -- (6,3);

        \node[squarednode] at (2.5,1.5) {$*$};
        \node[squarednode] at (2.5,2.5) {$*$};
        
        \node[squarednode] at (3.5,1.5) {$a'$};
        \node[squarednode] at (3.5,2.5) {$e$};
        
        \node[squarednode] at (4.5,2.5) {$*$};
        \node[squarednode] at (4.5,1.5) {$c$};
        \node at (3.5,4.3) {$i$};
    \end{tikzpicture}
\end{center}
where \eqref{eq:W,P 1} yields the placement of $d$, \eqref{eq:W,P 2} yields the placement of $a,b$, \eqref{eq:W,P 3} yields the placement of $c$, and finally, our inductive hypothesis yields the placement of $e$.  It now suffices to show that $a' = b'$.  

Since position $\Wcoor{k-1}{i}$ in $\W{}{n-i-1}(w)$ eventually contains a dot it follows that  $b' = \min(a,c)$.  Further, one of positions $\Wcoor{k-2}{i-1}$ or $\Wcoor{k-2}{i}$ in $\W{}{n-i-1}(w)$ must eventually contain a dot during this transformation and hence $d= b$ or $e =b$.   (Note when $k=2$ the transformation from $\W{}{n-i-1}(w)$ to $\W{}{n-i}(w)$ begins by placing dots at the top of columns $1,\ldots, i$.) 
 
 By definition of $\iRSK$ we have $a'\leq a,c$ and so $a' \leq \min(a,c)$. For a contradiction assume $a'<\min(a,c)$, where one or both of $a$ or $c$ might be $\infty$.   This means that $a'$ bumps $a$ and that this is the rightmost occurrence of $a'$ in the $k$th row of $\cP(w)$.  Consequently,  $b=a'$ and $d< b$.   As $\cP(w)$ is increasing along columns we have $e<a'=b$.  This contradicts the fact that either $b=d$ or $b=e$.  

If $i=1$ then $d$ is absent yet an argument analogous to that in the previous paragraph proves that $a' = b'$ as needed.  

\end{proof}

\begin{lemma}\label{lem:P = W}
	If $w\in \fW_n$ then 
	$$\W{i}{n-i}(w)=\cP_i(w)$$ 
	for all $i\leq n$.
\end{lemma}
\begin{proof}
We proceed by induction on $n$ and note that the result clearly holds when $n=1$. Now fix some $n>1$ and some word $w$ with length $n$.  Observe that when $i = n$ we have 
$$\W{n}{0}(w) = w_n = \cP_n(w).$$
In light of this we further proceed by (reverse) induction on $i$, the column index.   Fix $i<n$ and assume that
$$\W{i+1}{n-i-1}(w) = \cP_{i+1}(w).$$
By Lemma~\ref{lem:cols of W} and our inductive hypothesis on $n$ we have
$$\W{i}{n-i-1}(w) = \W{i}{n-i-1}(w_1\cdots w_{n-1}) = \cP_i(w_1\cdots w_{n-1})$$
and
$$\W{i-1}{n-i}(w) = \W{i-1}{n-i}(w_1\cdots w_{n-1}) = \cP_{i-1}(w_1\cdots w_{n-1}).$$
It now follows by Lemma~\ref{lem:W,P} that 
$$\W{i}{n-i}(w) = \cP_i(w)$$
and hence the result holds for all words of length $n$ as needed.  
\end{proof}

\begin{lemma}\label{lem:cM determines cW}
For $i+j\leq n$ we have
$$|\W{i}{j}(w)| = |\makeset{x\in [n-j,n]}{x\notin \cM_i(w) }|.$$
\end{lemma}

\begin{proof}
First observe that this statement trivially holds for $j=0$ since  $n\notin  \cM_i(w)$ and $\W{i}{0} = w_i$ for all $i\leq n$.  Proceeding inductively, assume the claim holds for some $j<n-i$ and consider the transformation from $\W{}{j}(w)$ to $\W{}{j+1}(w)$.  A consequence of $\W{}{j}(w)$ being an increasing tableau is that during this transformation there is at most one dot in the $i$th column at any given step.  Further, any time a dot enters this column the number of elements in this column decreases by exactly 1.  Likewise, whenever a dot exits this column the number of elements in this column increases by exactly 1. 

Now consider two cases.  First assume there is no dot in the $i$th column in $\kW{}{j}{n}(w)$.  This means that $n-(j+1)\notin \cM_i(w)$.  From the fact that $i< n-j$ (we assumed $j< n-i$ above), it follows that there is a dot in column $i$ in $\kW{}{j}{0}(w)$.    As we start with a dot in this column and end with no dot, it follows from the first paragraph that $|\W{i}{j+1}(w)|= 1+ |\W{i}{j}(w)|$. By induction we now have
$$|\makeset{x\in [n-(j+1),n]}{x\notin \cM_i(w) }| = 1+ |\makeset{x\in [n-j,n]}{x\notin \cM_i(w) }| = |\W{i}{j+1}(w)|.$$

Next, assume there is a dot in the $i$th column in $\kW{}{j}{n}(w)$ so that $n-(j+1)\in \cM_i(w)$. By reasoning similar to the preceding paragraph, we have $|\W{i}{j+1}(w)|= |\W{i}{j}(w)|$ and hence, by induction, 
$$|\makeset{x\in [n-(j+1),n]}{x\notin \cM_i(w) }| = |\makeset{x\in [n-j,n]}{x\notin \cM_i(w) }| = |\W{i}{j+1}(w)|.$$
This completes our proof.
\end{proof}

\begin{lemma}\label{lem:R -> Q}
Fix $u,w\in \fW_n$.  If $\cM(u) = \cM(w)$ then $\cQ(u) = \cQ(w)$. 
\end{lemma}

\begin{proof}
	Let $\lambda^{(j)}$ be the shape of $\cP(w_1\cdots w_j)$.  By our definition we know that $\cQ(w)$ is completely determined by the sequence $\emptyset \leq \lambda^{(1)}\leq \cdots \leq \lambda^{(n)}$.  Therefore to prove the lemma it suffices to show that $\cM(w)$ completely determines the $\lambda^{(j)}$'s.  
	
To this end observe that, by Lemma~\ref{lem:P = W}, Lemma~\ref{lem:cols of W} and Lemma~\ref{lem:cM determines cW}, the height of the $i$th column of $\lambda^{(j)}$, for all $i\leq j$, is given by
$$|\cP_i(w_1\cdots w_{j})| = |\W{i}{j-i}(w)| = |\makeset{x\in [n-j+i,n]}{x\notin \cM_i(w) }|.$$
In other words, $\lambda^{(j)}$ is determined by the first $j$ columns of $\cM(w)$.  This implies our result.    
\end{proof}

\subsection{$\cQ(u)=\cQ(w)$ implies $u\sim_{cd} w$ and $\cM(u)=\cM(w)$ implies $u\sim_{cid} w$ }\label{subsec:irsk3}

To accomplish the goals of this section, we shall need an alternative description of the iRSK algorithm.  Not surprisingly, this description will rest on well-known properties of classical $\RSK$.  We begin by recalling $\RSK$ and some of these classical properties.

First we recall that a \emph{2-line array} is an array of positive integers consisting of two rows of equal length, such that the columns are ordered lexicographically from left to right, with the top entries taking precedence.  We call the elements of the top row \emph{indices}, and those of the bottom row \emph{values}.  The \emph{length} of such an array is the length of these rows. 

For any 2-line array $A$ we define $A^{\mystar}$ to be the 2-line array obtained by switching the top and bottom rows of $A$ and then placing the columns in lexicographic order.

We next recall that $\RSK$ is a bijection between all 2-line arrays of length $n$ and all pairs of semistandard Young tableaux of size $n$ with the same shape.  Going forward we write 
 $$\RSK(A) = (P(A), Q(A))$$
 where $P(A)$ and $Q(A)$ are the usual insertion and recording tableaux, respectively.  Since any word $w$ of length $n$ can be viewed as a 2-line array with values $w_1,\ldots, w_n$ and indices $1,\ldots, n$, we can also write $P(w)$ and $Q(w)$. A fundamental connection between RSK and inverses of 2-line arrays is the following. (For a proof see Chapter 4.2 in \cite{Fulton1996}.)
 
 \begin{fact}\label{fact:inverse rsk}
 For any 2-line array $A$ with $\RSK(A) = (P,Q)$ we have 
 $$\RSK(A^{\mystar}) = (Q,P).$$
 \end{fact}
 
 There  is an important equivalence relation on words that is related to $\RSK$.  Recall that two words $u,w$ are said to be \emph{Knuth equivalent}, in which case we write $u\equiv w$, provided that one can be obtained from the other by a finite number of moves of the following two types:
\begin{align*}
	\cdots bca\cdots \qquad &\leftrightarrow \quad \cdots bac\cdots \qquad (a< b \leq c)\\
	\cdots acb\cdots \qquad &\leftrightarrow \quad \cdots cab\cdots\qquad (a \leq b < c).\\	
	\end{align*}
 Knuth equivalence and $\RSK$ are connected by the following classic result, which we shall need momentarily.  (For a proof see Theorem 3.4.3 in \cite{Sagan2001}.)

\begin{fact}\label{fact:Knuth implies P}
If $u,w$ are words, then $u\equiv w$ if and only if $P(u) = P(w)$. 
\end{fact}

With these basics in hand we now introduce a variant of Knuth equivalence that is tailored to our needs.  

\begin{definition}
	We say two words $u$ and $w$ are \emph{i-Knuth equivalent} and write $u\equiv_i w$ provided that one can be obtained from the other by a finite number of the following moves:
	\begin{align*}
	\cdots a\cdots \qquad &\leftrightarrow \quad \cdots aa\cdots\\
	\cdots bca\cdots \qquad &\leftrightarrow \quad \cdots bac\cdots \qquad (a< b \leq c)\\
	\cdots acb\cdots \qquad &\leftrightarrow \quad \cdots cab\cdots\qquad (a \leq b < c).\\	
	\end{align*}
	  We call the first type of move an \emph{inflation}, and we say that any two words \emph{differ in multiplicity} provided that they differ by a finite number of inflations. 
\end{definition}

For our purposes, there are two salient properties of this definition.  The first is the obvious fact that if $u\equiv w$ then $u\equiv_i w$. The second property is the subject of our next lemma.

\begin{lemma}\label{lem: rd}
If  $u\equiv_i w$ then $u\sim_{rd} w$.  
\end{lemma}
\begin{proof} We first show that if $u\equiv_i w$ then $\lds(u)=\lds(w)$.
This is clear if $u$ and $w$ differ by a single inflation. On the other hand, if $u\equiv w$  then, by Fact~\ref{fact:Knuth implies P}, $P(u) = P(w)$ so by Exercise 1 in Chapter 3 of \cite{Fulton1996}, $\lds(u) = \lds(w)$. So $u\equiv_i w$ implies that $\lds(u)=\lds(w)$.  

It now suffices to show that if $u \equiv_i w$ and $u'$ and $w'$ are obtained from $u$ and $w$, respectively, by deleting all occurrences of the largest (respectively, smallest) letter then $u' \equiv_i w'$. If $u$ and $w$ differ by a single inflation then either $u'=w'$ or $u'$ and $w'$ differ by a single inflation, so $u'\equiv_i w'$.  On the other hand, if $u$ and $w$ differ by a single Knuth move then by Lemma 3 in Chapter 3 of~\cite{Fulton1996} we have $u'\equiv w'$ and so $u'\equiv_i w'$.  
\end{proof}

In order to bring CID information into the picture we need the following definition.  

\begin{definition}
For any $w\in \fW_n$ we define $w^\mystar$ to be the word obtained as follows.  Let $A$ be the 2-line array representing $w$, i.e., its values are $w_1,\ldots, w_n$ and its indices are $1,\ldots, n$.  Define $w^\mystar$ to be the word given by the values of $A^\mystar$.    
\end{definition}

 For example, if $w = 4113252$, then 
 \renewcommand{\arraystretch}{1}
 $$A = \left(\begin{array}{ccccccc}
	1 & 2& 3&4&5&6&7\\
	4&1&1&3&2&5&2
\end{array}\right)\qquad\text{and}\qquad A^{\mystar} = \left(\begin{array}{ccccccc}
	1&1 & 2&2 & 3 & 4 & 5\\
	2&3 & 5&7 & 4 & 1 & 6
\end{array}\right),$$
giving $w^\mystar =2357416$.  Observe that $w^\mystar$ is always a permutation and so the map $w\mapsto w^\mystar$ is not an involution.

 To prove our main theorems we require two lemmas. 
  
  \begin{lemma}\label{lem:rd,cd}
 Let $u,w$ be words.  Then $u\sim_{cd} w$ if and only if $u^\mystar \sim_{rd} w^\mystar$.  
 \end{lemma}
 
 \begin{proof}
Fix an interval $[i,j]$.  Observe that we have a length-preserving correspondence between the decreasing sequences of $u_i\cdots u_j$ and the decreasing sequences of
$u^\mystar$ with values in $[i,j]$,
given by the map
$$u_{x_1}> \cdots > u_{x_k} \mapsto x_k>\cdots > x_1.$$
Note that this correspondence relies on the fact that in any 2-line array a strictly decreasing sequence is, by definition, indexed by a strictly increasing sequence. The lemma is now a consequence of this fact.  
 \end{proof}

 \begin{lemma}\label{lem:reading words}
 For every word $w$,
 $$w \equiv_i \Read(\cP(w))\myand w^\mystar \equiv_i \Read(\cQ(w)).$$
 \end{lemma}
 
We postpone the proof of this result since it is a bit involved and requires new ideas.  Instead we turn to the main theorems of the section.
 
 \begin{thm}\label{thm:Q implies cd}
 	Let $u,w\in \fW_n$.  If $\cQ(u) = \cQ(w)$ then $u \sim_{cd} w$.
 \end{thm}
 \begin{proof} 
 By Lemma~\ref{lem:reading words} we have
$$u^\mystar \equiv_i \Read(\cQ(u))\myand w^\mystar \equiv_i \Read(\cQ(w)).$$
Since $\cQ(u) = \cQ(w)$ we have $u^\mystar \equiv_i w^\mystar$, so by Lemma~\ref{lem: rd} we have  $u^\mystar \sim_{rd} w^\mystar$.  Lemma~\ref{lem:rd,cd} now implies $u \sim_{cd} w$, as desired.
 \end{proof}

\begin{remark}
In light of this theorem, a natural question is whether $\cQ$ captures CID data.  Unfortunately, it does not, as evidenced by the words $11$ and $u=12$ which are not CID equivalent but do have the same $\cQ$ tableau.  
\end{remark} 

We are now in position to prove Theorem~\ref{thm:M implies cid}.
\begin{proof}[Proof of Theorem \ref{thm:M implies cid}]  
If $\cM(u)=\cM(w)$ then by Lemma~\ref{lem:R -> Q} and Theorem~\ref{thm:Q implies cd} we have $u\sim_{cd} w$.

Now consider the reverse words $u^r$ and $w^r$.  By taking transposes we have $\cM(u^r)=\cM(w^r)$ and so $u^r \sim_{cd} w^r$ by the above argument.  Equivalently, we have  $u \sim_{ci} w$. Combining this with the result of the preceding paragraph gives $u\sim_{cid}w$ as needed.  
\end{proof}

 \subsubsection{Proof of Lemma~\ref{lem:reading words}}
 
To prove this lemma we make use of an alternative description of the iRSK algorithm, in terms of 2-line arrays.

\begin{definition}
For any 2-line array $A$ define the triple
 	$$f(A)= (B, p, q),$$
 	where $p$ and $q$ are the top rows of $P(A)$ and $Q(A)$, respectively, and $B = \RSK^{-1}( \widehat{P(A)}, \widehat{Q(A)})$ where $\widehat{T}$ denotes the tableau obtained from a tableau $T$ by deleting the top row.   	
\end{definition}

\begin{definition}
    Let $w\in \fW_n$ and consider the construction of the i-insertion tableau $\cP(w)$.  For each  $i> 0$ and $j\leq n,$ let $A_{ij}$ be the 2-line array whose values are given by the sequence of elements bumped from row $i$ during the insertion
    $$\cP(w_1\cdots w_{j-1})\leftarrow w_j$$
    and whose indices are $j$.  We then define 
    $$A_i = A_{i1}\cdots A_{in}$$
    and refer to this array as the \emph{$i$th bumped array for $w$}.  We  define $A_0$ to be the 2-line array given by the word $w$.  
    
    We make the same definition with $\RSK$ in place of $\iRSK$.  Context will make it clear which variant we are referencing.  
\end{definition}
For example, take the word $w = 5143642$ with $\RSK(w)$ given by
$$\ytableaushort{124,36,4,5} \qquad \ytableaushort{135,26,4,7}\ .$$
A simple check shows that the bumped arrays in this case are 
\renewcommand{\arraystretch}{1}
$$A_1 = \left(\begin{array}{ccccc}
	2 & 4 &6&7\\
	5&4&6&3
\end{array}\right)\qquad A_2 = \left(\begin{array}{cc}
	4&7\\
	5&4
\end{array}\right)\qquad A_3 = \left(\begin{array}{ccccc}
	7\\
	5
\end{array}\right).$$
Observe that the insertion and recording tableaux for $A_i$ are precisely the tableaux obtained by deleting the first $i$ rows of $P$ and $Q$, respectively.  For example $\RSK(A_1)$ gives 
$$\ytableaushort{36,4,5}\qquad \ytableaushort{26,4,7}\ .$$

To state our next result we shall need the idea of two 2-line arrays $A$ and $B$ \emph{differing in multiplicity}.  By this we simply mean that  $A$ can be obtained from $B$ by a finite number of the following moves:
\renewcommand{\arraystretch}{1}
	$$\left(\begin{array}{ccc} 
	\cdots & a&\cdots  \\ 
	\cdots & b&\cdots
	\end{array}\right)  \qquad \leftrightarrow \qquad 
	\left(\begin{array}{cccc} 
	\cdots & a & a&\cdots  \\ 
	\cdots & b & b&\cdots\end{array}\right).$$  
Note that if $A$ and $B$ differ in multiplicty then the words given by their values are i-Knuth equivalent.

\begin{lemma}\label{lem:exist A's}
Fix a word $w$ of length $n$. Let $A_0,\ldots, A_k$ be the  bumped arrays  from the construction of $\cP(w)$ and let the $i$th rows of $\cP(w)$ and $\cQ(w)$ be $p_i$ and $q_i$, respectively.  Then there exist 2-line arrays $A_0^+,\ldots, A_k^+$ such that $A_i$ and $A_i^+$ differ in multiplicity and 
 $$f(A_i^+) = (A_{i+1}, p_{i+1},q_{i+1}).$$
     	
\end{lemma}

\begin{proof}
By definition of $\iRSK$ there exist arrays $A_i^+$ which differ in multiplicity from $A_i$ so that if 
$$\RSK(A_i^+) = (P,Q)$$
then $p_{i+1}$ and $q_{i+1}$ are the top rows of $P$ and $Q$, respectively, and the first bumped array from the construction of $P(A_i^+)$ is $A_{i+1}$.  In other words,
$$\RSK(A_{i+1}) = (\widehat P,\widehat Q)$$
and so
$f(A_i^+) = (A_{i+1},p_{i+1}, q_{i+1})$ as needed.   
\end{proof}

To prove our next lemma we make use of the following well-known fact.  A proof of this can be found in Proposition 1 of Chapter 2 in \cite{Fulton1996}.  

\begin{fact}\label{fact: w equiv reading(P)}
For any word $w$, we have $w\equiv \Read(P(w))$.
\end{fact}
Since Knuth equivalence implies i-Knuth equivalence, it follows immediately that $w\equiv_i \Read(P(w))$.  We make use of this fact in this form below.

\begin{lemma}\label{lem:f and K-Knuth}
	Let $A$ be a 2-line array and set $f(A) = (B,p,q)$. Denote by $\alpha$ and $\beta$ the words given by the values of $A$ and $B$, respectively.  Then we have
\begin{equation}\label{eq:word of A}
\alpha \equiv_i \beta\cdot p
\end{equation}
and 
\begin{equation}\label{eq:inverse of A}
	f(A^{\mystar}) = (B^{\mystar}, q,p).	
\end{equation}
\end{lemma}
\begin{proof}
	Let $\RSK(A)=(P,Q)$ so that $P(\alpha) = P$. So $\alpha \equiv_i \Read(P)$.    By definition of $B$ we also know that $P(\beta) = \widehat P$ and so  $\beta \equiv_i \Read(\widehat P)$.    Consequently,
$$\alpha \equiv_i \Read(P) = \Read(\hat{P})\cdot p \equiv_i \beta \cdot p,$$
which proves the first claim.  The second claim follows immediately by applying Fact~\ref{fact:inverse rsk}, as this shows that $\RSK(A^{\mystar}) = (Q,P)$ and $\RSK(B^{\mystar}) = (\widehat{Q},\widehat{P})$.  
\end{proof}

\begin{proof}[Proof of Lemma~\ref{lem:reading words}]
 \sloppy To prove the first claim let $A_0,\ldots, A_k$ and $A_0^+,\ldots, A_k^+$ be as in Lemma~\ref{lem:exist A's} and for $0\leq j\leq k$ let $\alpha_j$ and $\alpha_j^+$ be the words obtained from the values of $A_j$ and $A_j^+$, respectively. So 
 $$f(A_j^+) = (A_{j+1},p_{j+1},q_{j+1}),$$
 where $p_j$ and $q_j$ are the $j$th rows of $\cP(w)$ and $\cQ(w)$, respectively.  Furthermore, $\alpha_j$ and $\alpha_j^+$ differ in multiplicity and so $\alpha_j \equiv_i \alpha_j^+$.
 By repeated applications of \eqref{eq:word of A} in Lemma~\ref{lem:f and K-Knuth} we have
 $$w \equiv_i \alpha_0^+ \equiv_i \alpha_1p_1\equiv_i \alpha_1^+p_1\equiv_i \alpha_2p_2p_1\equiv_i \cdots \equiv_i p_k\cdots p_1 = \Read(\cP(w)).$$
 
 To prove the second claim set $C_j = A_j^{\mystar}$ and $C_j^+ = (A_j^+)^{\mystar}$.  By Equation~\eqref{eq:inverse of A} in Lemma~\ref{lem:f and K-Knuth} we have 
 $$f(C_j^+) = (C_{j+1}, q_{j+1},p_{j+1}).$$
Since $w^\mystar$ is the word given by the values in $C_0= A_0^{\mystar}$,  an argument analogous to that given for the first claim proves the second claim. 
\end{proof}

\section{The conjecture of Guo and Poznanovi\'{c}}\label{sec:proof of conjecture}

The purpose of this section is to prove the conjecture of Guo and Poznanovi\'{c} (\cite{GuoPoznanovik2020}, Conjecture 15), concerning the lengths of the longest ne and se chains in certain 01-fillings of moon polyominoes.  For the convenience of the reader, we recall that $N(M;n;ne=u,se=v)$ denotes the number of 01-fillings of $M$ with exactly $n$ 1's such that each column contains at most one 1, and the lengths of the longest ne and se  chains in $M$ are $u$ and $v$, respectively.

\begin{proof}[Proof of Theorem~\ref{thm:moon}]
We will prove that, starting with any moon polyomino $M$, we can perform a sequence of transformations of $M$ that turn it first into a stack polyomino and then into a Ferrers board in French notation, preserving the multiset of row lengths and $N(M; n; ne = u; se = v)$ at each step.  Since a Ferrers board in French notation is completely determined by its multiset of row lengths, this will prove the theorem. 

To begin, let $M$ be a moon polyomino and let $c$ be the leftmost column of $M$.  Let $R$ be the largest rectangle in $M$ that contains column $c$ and let $M'$ be the moon polyomino obtained from $M$ by deleting $c$ and adding a new column $c'$ at the right end of $R$.  Let $R'$ be the rectangle in $M'$ that has the same size as $R$, so that $R'$ is ``$R$ shifted one unit to the right."  
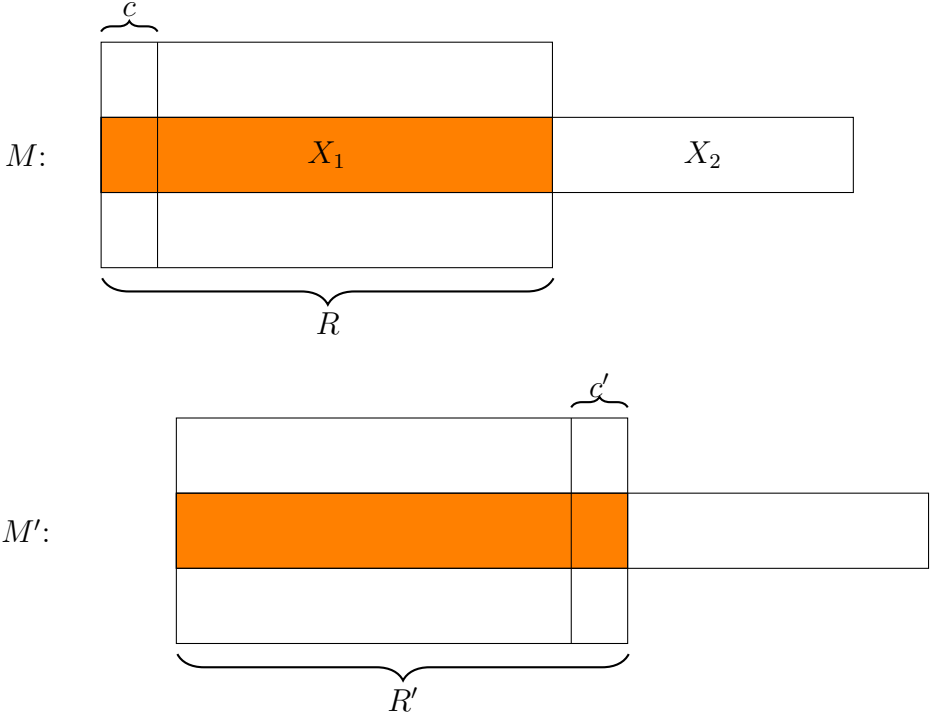
\begin{center}
	\begin{tikzpicture}
	\node at (-1,1.5) {$M$:};
		\draw[fill=orange]  (0,1) rectangle (6,2);
		\draw (0,0) rectangle (6,3);
		\draw (0,1) rectangle (10,2);
		\draw (.75,0) -- (.75,3);
		\node at (3,1.5) {$X_1$};
		\node at (8,1.5) {$X_2$};
		\draw [thick,decorate,decoration={brace,amplitude=10pt,mirror},xshift=0.4pt,yshift=-4pt](0,0) -- (6,0) node[black,midway,yshift=-0.6cm] {$R$};
		
		\draw [thick,decorate,decoration={brace,amplitude=4pt},xshift=0pt,yshift=4pt](0,3) -- (.75,3) node[black,midway,yshift=0.3cm] {$c$};

		\begin{scope}[shift={(1,-5)}]
		\node at (-2,1.5) {$M'$:};
		\draw[fill=orange]  (0,1) rectangle (6,2);
		\draw (0,0) rectangle (6,3);
		\draw (0,1) rectangle (10,2);
		\draw (5.25,0) -- (5.25,3);

		\draw [thick,decorate,decoration={brace,amplitude=10pt,mirror},xshift=0.4pt,yshift=-4pt](0,0) -- (6,0) node[black,midway,yshift=-0.6cm] {$R'$};
		
		\draw [thick,decorate,decoration={brace,amplitude=4pt},xshift=0pt,yshift=4pt](5.25,3) -- (6,3) node[black,midway,yshift=0.3cm] {$c'$};
  \end{scope}
	\end{tikzpicture}
\end{center}
It is clear that $M$ and $M'$  have the same multiset of row lengths.  We will show that 
\begin{equation}\label{eq:ne,se}
N(\cM; n; ne = u, se = v) = N(\cM'; n; ne = u, se = v).	
\end{equation}

To each 01-filling of $\cM$ that has at most one 1 in each column we associate a corresponding filling of $\cM'$ as follows.  For boxes of $M$ outside $R$ we retain the given filling.  We modify the filling of $R$ to obtain a filling of $R'$ as follows.  If $c$ is empty we leave $c'$ empty and retain the given filling of the rest of $R$. If $c$ is not empty, we let $w$ be the word whose matrix representation is given by the nonempty rows and columns of the given filling of $R$ and we take the word $w'$ associated to $w$ by Theorem~\ref{thm:col shift}, recalling that the values that occur in $w'$ are precisely those that occur in $w$. We place the matrix representation of $w'$ in the columns of $R'$ that were not empty in $R$, and $c'$.  This leaves empty the rows of $R$ that were empty in the given filling of $M$.  

To prove \eqref{eq:ne,se} it suffices to show that for every 01-filling of $M$ with at most one 1 in each column, and any maximal rectangle $X$ of $M$, the length of the longest ne (respectively, se) chains in $X$ is the same as the length of the longest ne (respectively, se) chains in the corresponding maximal rectangle $X'$ of $M'$ under the associated filling of $M'$.  For $R$ and $R'$ this is clear, since the bijection of Theorem~\ref{thm:col shift} preserves the lengths of the longest increasing and decreasing sequences in a word.

Now suppose $X$ is a maximal rectangle wider than $R$, so that $X$ consists of a rectangle $X_1$ contained in $R$ with another rectangle $X_2$ of the same height adjoined to its righthand side, as shown above.

Any ne chain of maximal length $\ell$ in $X$ consists of an initial segment contained in some interval $I$ of the rows of $R$, and a terminal segment in $X_2$.  The initial segment must have maximal length among ne chains in the interval $I$ of rows of $R$, so by Theorem~\ref{thm:col shift} the interval $I$ of rows of $R'$ contains a ne chain of the same length.  This sequence, adjoined to the terminal segment in $X_2$, gives a ne chain of length $\ell$ in $X'$.  A similar argument applies to se chains. If we then repeat the argument, interchanging the roles of $X'$ and $X$, we conclude that the lengths of the longest ne and se chains in $X$ and $X'$ coincide.  

If $X$ is a maximal rectangle narrower than $R$, then a ne chain of maximal length $\ell$ in $X$ consists of initial and terminal segments outside $R$ and a middle segment contained in some interval $J$ of the columns of $R$. The middle segment must have maximal length among ne chains in the interval $J$ of columns of $R$, so by Theorem~\ref{thm:col shift} the interval $J$ of columns of $R'$ contains a ne chain of the same length. This chain, together with the initial and terminal segments outside $R$, gives us a ne chain of length $\ell$ in $X'$.  As in the preceding paragraph, we can now conclude that the lengths of the longest ne and se chains in $X$ and $X'$ coincide.  

By iterating the argument contained in the preceding five paragraphs, we can convert $M$ into a stack polyomino $M^\diamond$, preserving the multiset of row lengths and $N(M;n;ne=u,se=v)$.  To complete the proof, we can invoke the result of Guo and Poznanovi\'{c} (\cite{GuoPoznanovik2020}, Theorem 1), or we can convert the stack polyomino into a Ferrers board in French notation by doing an argument similar to the above, starting with the bottom row $r$ of $M^\diamond$ instead of the left column, taking the largest rectangle in $M^\diamond$ that contains $r$, and adding a row $r'$ to the top of this rectangle and deleting $r$.  We then use Theorem~\ref{thm:row shift} in place of Theorem~\ref{thm:col shift} to complete an argument similar to the above. (In this case, when we place the word $w'$ in the columns of $R'$ that were not empty in $R$, we must be careful to place the entries, from smallest to largest, into the rows of $R'$ that were not empty in $R$, and $r'$.)  
\end{proof}

\bibliographystyle{plain} 
\bibliography{mybib}
\end{document}